\documentclass[11pt,a4paper]{amsart}

\newcommand{\C}{\mathbb{C}}
\newcommand{\R}{\mathbb{R}}
\newcommand{\N}{\mathbb{N}}

\usepackage{amssymb,amsmath,amsthm,enumerate}
\usepackage{graphicx}
\usepackage{color}

\usepackage{hyperref}
\usepackage{cite}

\usepackage{etoolbox}
\newtheorem{theorem}{Theorem}[section]
\newtheorem{lemma}[theorem]{Lemma}
\newtheorem{proposition}[theorem]{Proposition}
\newtheorem{corollary}[theorem]{Corollary}
\newtheorem{definition}[theorem]{Definition}

\preto{\section}{}
\preto{\subsection}{}

\makeatletter
\@addtoreset{theorem}{section}
\@addtoreset{theorem}{subsection}
\makeatother

\usepackage{url}
\begin{document}

\title{Counting multijoints}
\author{ Marina Iliopoulou}
\address{School of Mathematics, University of Birmingham, Birmingham, Edgbaston, B15 2TT, UK}
\email{\href{mailto:M.Iliopoulou@bham.ac.uk}{M.Iliopoulou@bham.ac.uk}}
\begin{abstract} Let $\mathfrak{L}_1$, $\mathfrak{L}_2$, $\mathfrak{L}_3$ be finite collections of $L_1$, $L_2$, $L_3$, respectively, lines in $\R^3$, and $J(\mathfrak{L}_1, \mathfrak{L}_2,\mathfrak{L}_3)$ the set of multijoints formed by them, i.e. the set of points $x \in \R^3$, each of which lies in at least one line $l_i \in \mathfrak{L}_i$, for all $i=1,2,3$, such that the directions of $l_1$, $l_2$ and $l_3$ span $\R^3$. We prove here that $|J(\mathfrak{L}_1, \mathfrak{L}_2,\mathfrak{L}_3)|\lesssim (L_1L_2L_3)^{1/2}$, and we extend our results to multijoints formed by real algebraic curves in $\R^3$ of uniformly bounded degree, as well as by curves in $\R^3$ parametrised by real univariate polynomials of uniformly bounded degree. The multijoints problem is a variant of the joints problem, as well as a discrete analogue of the endpoint multilinear Kakeya problem.
\end{abstract}
\maketitle

\section{Introduction}

Let $\mathfrak{L}_1, \ldots , \mathfrak{L}_n$ be finite collections of lines in $\R^n$. We say that a point $x \in \R^n$ is a \textit{multijoint} formed by the $n$ collections of lines if, for each $i=1,\ldots, n$, there exists a line $l_i \in \mathfrak{L}_i$ passing though $x$, so that the directions of $l_1, \ldots, l_n$ span $\R^n$. We denote by $J(\mathfrak{L}_1,\ldots,\mathfrak{L}_n)$ the set of multijoints formed by $\mathfrak{L}_1,\ldots,\mathfrak{L}_n$.

The multijoints problem lies in bounding from above the number of multijoints by a number depending only on the dimension and the cardinalities of the collections of lines forming them. In particular, the multijoints problem is a variant of the well-known joints problem.

More specifically, a point $x \in \R^n$ is a joint for a collection $\mathfrak{L}$ of lines in $\R^n$ if there exist at least $n$ lines in $\mathfrak{L}$ passing through $x$, whose directions span $\R^n$. Note that a multijoint formed by $n$ collections of lines in $\R^n$ is a joint formed by the union of the collections. Now, the joints problem, i.e. the problem of bounding the number of joints by a power of the number of the lines forming them, first appeared in \cite{Chazelle_Edelsbrunner_Guibas_Pollack_Seidel_Sharir_Snoeyink_1992}, where it was proved that if $J$ is the set of joints formed by a collection of $L$ lines in $\R^3$, then $|J| \lesssim L^{7/4}$. Successive progress was made in improving the upper bound of $|J|$ in three dimensions, by Sharir, Sharir and Welzl, and Feldman and Sharir (see \cite{MR1280600}, \cite{MR2047237}, \cite{MR2121298}). Wolff  had already observed in \cite{MR1660476} that there exists a connection between the joints problem and the Kakeya problem, and, using this observation, Bennett, Carbery and Tao found an improved upper bound for $|J|$, with a particular assumption on the angles between the lines forming each joint (see \cite{MR2275834}). 

Eventually, Guth and Katz provided a sharp upper bound in \cite{Guth_Katz_2008}; they showed that, in $\R^3$, $|J|\lesssim L^{3/2}$. The proof was an adaptation of Dvir's algebraic argument  in \cite{MR2525780} for the solution of the finite field Kakeya problem, which involves working with the zero set of a polynomial. The contribution by Dvir, Guth and Katz was very significant, because they established the use of the polynomial method in incidence geometry. Further work was done by Elekes, Kaplan and Sharir in \cite{MR2763049}, and finally, a little later, Kaplan, Sharir and Shustin (in \cite{MR2728035}) and Quilodr\'{a}n (in \cite{MR2594983}) independently solved the joints problem in all dimensions, using again algebraic techniques, simpler than in \cite{Guth_Katz_2008}.

In particular, Quilodr\'{a}n and Kaplan, Sharir and Shustin showed that, if $\mathfrak{L}$ is a collection of $L$ lines in $\R^n$, $n \geq 2$, and $J$ is the set of joints formed by $\mathfrak{L}$, then 
\begin{equation}|J| \lesssim_n L^{\frac{n}{n-1}}. \label{eq:basic} \end{equation} 
In the light of this, Carbery has conjectured that, if, for each joint $x \in J$, the multiplicity $M(x)$ of $x$ is the number of $n$-tuples of linearly independent lines of $\mathfrak{L}$ passing through $x$, then
\begin{equation}\label{eq:jointsmultiplicities}\sum_{x \in J}M(x)^{1/(n-1)}\lesssim_n L^{n/(n-1)};
\end{equation} in fact, we have shown \eqref{eq:jointsmultiplicities} for $n=3$ in \cite{Iliopoulou_12}.

Now, if $\mathfrak{L}_1, \ldots, \mathfrak{L}_n$ are finite collections of $L_1, \ldots, L_n$, respectively, lines in $\R^n$, for each multijoint $x \in J(\mathfrak{L}_1,\ldots ,\mathfrak{L}_n)$ we define the multiplicity of $x$ as $N(x):=\{(l_1,\ldots,l_n) \in \mathfrak{L}_1\times\cdots\times\mathfrak{L}_n$: $x \in l_i$ for all $i=1,\ldots,n$, and the directions of the lines $l_1, \ldots, l_n$ span $\R^n\}$.

In analogy to \eqref{eq:basic} and \eqref{eq:jointsmultiplicities}, Carbery has also conjectured that, for all $n \geq 3$ (for $n=2$ it is obvious),
\begin{equation} \label{eq:mult1}|J(\mathfrak{L}_1,\ldots,\mathfrak{L}_n)| \lesssim_n (L_1\cdots L_n)^{1/(n-1)} \end{equation}
and
\begin{equation} \label{eq:mult1000}\sum_{x \in J(\mathfrak{L}_1,\ldots,\mathfrak{L}_n)}N(x)^{1/(n-1)} \lesssim_n (L_1\cdots L_n)^{1/(n-1)}. \end{equation}

The aim of this paper is to show \eqref{eq:mult1} for $n=3$, i.e. the following.

\begin{theorem}\label{theoremmult2} Let $\mathfrak{L}_1$, $\mathfrak{L}_2$, $\mathfrak{L}_3$ be finite collections of $L_1$, $L_2$ and $L_3$, respectively, lines in $\R^3$. Then,
\begin{equation} \label{eq:proved}|J(\mathfrak{L}_1,\mathfrak{L}_2,\mathfrak{L}_3)|\leq c\cdot  (L_1L_2L_3)^{1/2}, \end{equation}
where $c$ is a constant independent of $\mathfrak{L}_1$, $\mathfrak{L}_2$ and $\mathfrak{L}_3$.
\end{theorem}

We would like to mention here that the multijoints problem, as the joints problem, can be seen from a harmonic analytic point of view; in fact, it is a discrete analogue of the endpoint multilinear Kakeya problem. More specifically, the multilinear Kakeya problem asks for the optimal upper bound of the volume of the intersection of $n$ essentially transversal families of tubes in $\R^n$, depending only on the cardinalities of the families and the number of tubes of each family passing through each point of $\R^n$ (where the expression ``essentially transversal" means that, for all $i=1,\ldots,n$, the direction of each tube in the family $\mathbb{T}_i$ lies in a fixed $\frac{c}{n}$-cap around the vector $e_i \in \R^n$, where the vectors $e_1, \ldots, e_n$ are orthonormal). Using unexpected techniques from algebraic geometry, Guth, in \cite{Guth_10}, improved the already existing result by Bennett, Carbery and Tao (see \cite{MR2275834}), and showed that, whenever $\mathbb{T}_1, \ldots, \mathbb{T}_n$ are $n$ essentially transversal families of doubly-infinite tubes in $\mathbb{R}^n$, with cross section an $(n-1)$-dimensional unit ball, it holds that
\begin{displaymath}\int_{x \in \R^n}\prod_{i=1}^{n}\Bigg(\sum_{T_i \in \mathbb{T}_i}\chi_{T_i}(x)\Bigg)^{1/(n-1)}{\rm d} x \lesssim_{n} (|\mathbb{T}_1| \cdots |\mathbb{T}_n|)^{1/(n-1)}, \end{displaymath} i.e.
\begin{equation}\label{eq:guthstrong}\int_{x \in \R^n}\big(\#\{\text{tubes of }\mathbb{T}_1\text{ through }x\}\cdots \#\{\text{tubes of }\mathbb{T}_n\text{ through }x\}\big)^{1/(n-1)}{\rm d}x \end{equation}
\begin{displaymath} 
\lesssim_n (|\mathbb{T}_1| \cdots |\mathbb{T}_n|)^{1/(n-1)},\end{displaymath}
an estimate that constitutes a continuous analogue of \eqref{eq:mult1000}. In addition, if $E(\mathbb{T}_1,\cdots,\mathbb{T}_n)$ is the set of points in $\R^n$ each lying in at least one tube of the family $\mathbb{T}_i$ for all $i=1,\ldots,n$, \eqref{eq:guthstrong} implies that
\begin{equation} \label{eq:guthsimple} vol_n(E(\mathbb{T}_1,\ldots,\mathbb{T}_n)) \lesssim_n (|\mathbb{T}_1| \cdots |\mathbb{T}_n|)^{1/(n-1)},
\end{equation} 
an estimate whose discrete analogue is clearly \eqref{eq:mult1}.

In order to prove Theorem \ref{theoremmult2}, we use algebraic methods similar to the ones developed by Guth and Katz in \cite{Guth_Katz_2010}. In fact, we deduce Theorem \ref{theoremmult2} as a corollary of Theorem \ref{multsimple} which follows.

\begin{definition} We say that $n$ collections $\mathfrak{L}_1, \ldots, \mathfrak{L}_n$ of lines in $\R^n$ are \emph{transversal} if, whenever $l_i \in\mathfrak{L}_i$, for $i=1,\ldots,n$, the directions of $l_1, \ldots, l_n$ span $\R^n$.

\end{definition}

\begin{definition} Let $\mathfrak{L}_1,\ldots,\mathfrak{L}_n$ be finite collections of lines in $\R^n$.

For all $(N_1,\ldots,N_n) \in \R_+^n$, we define as $J_{N_1,\ldots,N_n}(\mathfrak{L}_1,\ldots,\mathfrak{L}_n)$ the set of those multijoints $x$ formed by $\mathfrak{L}_1, \ldots, \mathfrak{L}_n$, with the property that there exist transversal collections $\mathfrak{L}_1(x) \subseteq \mathfrak{L}_1$, $\ldots$, $\mathfrak{L}_n(x) \subseteq \mathfrak{L}_n$ of lines passing through $x$, such that $|\mathfrak{L}_i(x)|\geq N_i$, for all $i=1,\ldots,n$.

\end{definition}

\begin{theorem} \label{multsimple} Let $\mathfrak{L}_1$, $\mathfrak{L}_2$, $\mathfrak{L}_3$ be finite collections of $L_1$, $L_2$ and $L_3$, respectively, lines in $\R^3$. Then,
\begin{displaymath} |J_{N_1,N_2,N_3}(\mathfrak{L}_1,\mathfrak{L}_2,\mathfrak{L}_3)|\leq c \cdot \frac{(L_1L_2L_3)^{1/2}}{(N_1N_2N_3)^{1/2}}, \; \forall\;(N_1,N_2,N_3) \in \N_{+}^3, \end{displaymath}
where $c$ is a constant independent of $\mathfrak{L}_1$, $\mathfrak{L}_2$ and $\mathfrak{L}_3$.
\end{theorem}

Indeed, Theorem \ref{theoremmult2} follows from Theorem \ref{multsimple}, as, for any finite collections $\mathfrak{L}_1$, $\mathfrak{L}_2$ and $\mathfrak{L}_3$ of lines in $\R^3$, and for each $x \in J(\mathfrak{L}_1,\mathfrak{L}_2,\mathfrak{L}_3)$, there exist lines $l_1(x) \in \mathfrak{L}_1$, $l_2(x) \in \mathfrak{L}_2$ and $l_3(x) \in \mathfrak{L}_3$, passing through $x$, with the property that their directions span $\R^3$. Therefore, $J(\mathfrak{L}_1,\mathfrak{L}_2,\mathfrak{L}_3)=J_{1,1,1}(\mathfrak{L}_1,\mathfrak{L}_2,\mathfrak{L}_3)$, and thus $|J(\mathfrak{L}_1,\mathfrak{L}_2,\mathfrak{L}_3)| \lesssim \frac{(L_1L_2L_3)^{1/2}}{(1 \cdot 1 \cdot 1)^{1/2}}\sim (L_1L_2L_3)^{1/2}$.

In fact, however, Theorem \ref{multsimple}, albeit seemingly stronger than Theorem \ref{theoremmult2}, is actually equivalent to Theorem \ref{theoremmult2}. We will explain this using a standard probabilistic argument (we were first introduced to the argument by Solymosi, via verbal communication with Carbery; it appears, for example, in \href{http://math.mit.edu/~lguth/PolyMethod/lect34.pdf}{``The multilinear Kakeya inequality"} online notes by Guth, while a detailed approach can be found in \cite[Chapter 7]{Iliopoulou_13}).\footnote{The reader may wonder at this point why we choose in this paper to first prove the more complicated statement of Theorem \ref{multsimple}, and then obtain the statement of Theorem \ref{theoremmult2} as a corollary, rather than just prove the simpler statement of Theorem \ref{theoremmult2} from scratch, and derive the statement of Theorem \ref{multsimple} as a corollary. The reason is that the proof we will present here for the seemingly more general statement of Theorem \ref{multsimple}, which includes all $(N_1,N_2,N_3) \in \N_{+}^3$, is not any less straightforward than its proof for the case $(N_1,N_2,N_3)=(1,1,1)$, i.e. than the proof of Theorem \ref{theoremmult2} itself; therefore, we believe that the main ideas behind the proof of Theorem \ref{theoremmult2} will be more easily demonstrated via the more general context of the proof of Theorem \ref{multsimple}.}

\textbf{Remark.} We hope that Theorem \ref{multsimple} may lead to the proof of \eqref{eq:mult1000}. In fact, even though we have not yet shown \eqref{eq:mult1000} in $\R^3$, we have managed to take advantage of Theorem \ref{multsimple} in \cite{Iliopoulou_13}, to effectively show \eqref{eq:mult1000} in $\R^3$ when the three collections of lines involved are transversal. 

More particularly, we prove in \cite{Iliopoulou_13} that, if $\mathfrak{L}_1$, $\mathfrak{L}_2$, $\mathfrak{L}_3$ are transversal, finite collections of $L_1$, $L_2$ and $L_3$, respectively, lines in $\R^3$, and, for each $x \in J(\mathfrak{L}_1,\mathfrak{L}_2, \mathfrak{L}_3)$ and $i=1,2,3$, $N_i(x)$ denotes the number of lines of $\mathfrak{L}_i$ passing through $x$, then
$$\sum_{\{x \in J(\mathfrak{L}_1,\mathfrak{L}_2,\mathfrak{L}_3):\;N_m(x)>10^{12}\}} \big(N_1(x)N_2(x)N_3(x)\big)^{1/2} \lesssim (L_1L_2L_3)^{1/2},
$$
where $m \in \{1,2,3\}$ is such that $L_m=\min\{L_1,L_2,L_3\}$. We have not yet managed, however, to show that$$\sum_{x \in J(\mathfrak{L}_1,\mathfrak{L}_2,\mathfrak{L}_3)} \big(N_1(x)N_2(x)N_3(x)\big)^{1/2} \lesssim (L_1L_2L_3)^{1/2},
$$
i.e. \eqref{eq:mult1000} in $\R^3$ when the three collections of lines involved are transversal, so we will not focus on this result here; details can be found in \cite{Iliopoulou_13}.

$\;\;\;\;\;\;\;\;\;\;\;\;\;\;\;\;\;\;\;\;\;\;\;\;\;\;\;\;\;\;\;\;\;\;\;\;\;\;\;\;\;\;\;\;\;\;\;\;\;\;\;\;\;\;\;\;\;\;\;\;\;\;\;\;\;\;\;\;\;\;\;\;\;\;\;\;\;\;\;\;\;\;\;\;\;\;\;\;\;\;\;\;\;\;\;\;\;\;\;\;\;\;\;\;\;\;\;\;\;\;\;\;\;\;\;\;\blacksquare$

Our techniques in this paper are based on the topology and continuous nature of euclidean space, as well as theorems that rely on them, such as the Borsuk-Ulam theorem and the intermediate value theorem, which is the reason why their immediate application to different field settings seems unlikely. In addition, we are taking advantage of the fact that the number of critical lines of an algebraic hypersurface in $\R^3$ is bounded, something whose truth we do not know in higher dimensions; therefore, we cannot yet apply our methods to solve the multijoints problem in $\R^n$, for $n\geq 4$.

As we have already mentioned, our basic tool will be the Guth-Katz polynomial method, as it appears in \cite{Guth_Katz_2010}. In Section \ref{section3} we go on to present this method, together with certain computational geometric results which will prove useful to our goal. In Section \ref{section4} we prove Theorem \ref{multsimple} and derive Theorem \ref{theoremmult2} as a corollary, while we also show that they are actually equivalent. Finally, in Section \ref{section5} we establish the geometric background appropriate for the study of the multijoints problem in the case where the multijoints are created by real algebraic curves, to eventually prove, in Section \ref{section6}, the corresponding statements of Theorem \ref{multsimple} and Theorem \ref{theoremmult2} for real algebraic curves in $\R^3$ of uniformly bounded degree (Theorem \ref{guthweakcurves} and Theorem \ref{theoremmult3}), as well as for curves in $\R^3$ parametrised by real univariate polynomials of uniformly bounded degree (Corollary \ref{polynomialcurves} and Corollary \ref{weakpolynomialcurves}).

We clarify here that, in whatever precedes and follows, any expression of the form $A \lesssim B$ means that there exists a non-negative constant $M$, depending only on the dimension, such that $A \leq M \cdot B$, while any expression of the form $A\lesssim_{b_1,\ldots,b_m} B$ means that there exists a non-negative constant $M_{b_1,\ldots,b_m}$, depending only on the dimension and $b_1, \ldots, b_m$, such that $A \lesssim M_{b_1,\ldots,b_m}\cdot B$. In addition, any expression of the form $A \gtrsim B$ or $A \gtrsim_{b_1,\ldots,b_m} B$ means that $B \lesssim A$ or $B \lesssim_{b_1,\ldots,b_m} A$, respectively. Finally, any expression of the form $A \sim B$ means that $A \lesssim B$ and $A \gtrsim B$, while any expression of the form $A \sim_{b_1,\ldots,b_m} B$ means that $A \lesssim_{b_1,\ldots,b_m} B$ and $A \gtrsim_{b_1,\ldots,b_m} B$.

\textbf{Acknowledgements.} I would like to thank my PhD supervisor, Anthony Carbery, for introducing me to the questions dealt with in this paper, for sharing with me his conjectures, thoughts and intuition, and for providing me with his support. While writing this paper, I was first supported partially by an EPSRC Doctoral Training Grant, and later by European Research Council Grant 307617, for both of which I am very grateful.

\section{Computational results on algebraic hypersurfaces}\label{section3}

Given a finite set of points $\mathfrak{G}$ in $\R^n$ and a quantity $d>1$, the Guth-Katz polynomial method results in a decomposition of $\R^n$, and consequently of the set $\mathfrak{G}$, by the zero set of a polynomial. Such a decomposition enriches our setting with extra structure, allowing us to derive information about the set $\mathfrak{G}$. The method is fully explained in \cite{Guth_Katz_2010}; here, we are presenting the basic result (which is based on the "ham sandwich" theorem by Stone and Tukey, see \cite{MR0007036}).

\begin{theorem}\emph{\textbf{(Guth, Katz, \cite[Theorem 4.1]{Guth_Katz_2010})}}\label{2.1.3} Let $\mathfrak{G}$ be a finite set of $S$ points in $\R^n$, and $d>1$. Then, there exists a non-zero polynomial $p \in \R[x_1,\ldots,x_n]$, of degree $\leq d$, whose zero set decomposes $\R^n$ in $\sim d^n$ cells, each of which contains $\lesssim S/d^n$ points of $\mathfrak{G}$.\end{theorem}

Decomposing $\R^n$ with the zero set of a polynomial immediately gives us a control over many quantities, especially in three dimensions. In particular, the following holds.

\begin{theorem}\emph{\textbf{(Guth, Katz, \cite[Corollary 2.5]{Guth_Katz_2008})}}\label{2.2.1} \emph{(Corollary of B\'ezout's theorem)} Let $p_1$, $p_2 \in \R[x,y,z]$. If $p_1$, $p_2$ do not have a common factor, then there exist at most $\deg p_1 \cdot \deg p_2$ lines simultaneously contained in the zero set of $p_1$ and the zero set of $p_2$. \end{theorem}

An application of this result enables us to bound the number of critical lines of a real algebraic hypersurface in $\R^3$.

\begin{definition} Let $p \in \R[x,y,z]$ be a non-zero polynomial of degree $\leq d$. Let $Z$ be the zero set of $p$. 

We denote by $p_{sf}$ the square-free polynomial we end up with, after eliminating all the squares appearing in the expression of $p$ as a product of irreducible polynomials in $\R[x,y,z]$.

A \emph{critical point} $x$ of $Z$ is a point of $Z$ for which $\nabla{p_{sf}}(x)=0$. Any other point of $Z$ is called a \emph{regular point} of $Z$. A line contained in $Z$ is called a \emph{critical line} if each point of the line is a critical point of $Z$.
\end{definition}

Note that, for any $p \in \R[x,y,z]$, the polynomials $p$ and $p_{sf}$ have the same zero set.

Moreover, if $x$ is a regular point of the zero set $Z$ of a polynomial $p \in \R[x,y,z]$, then, by the implicit function theorem, $Z$ is a manifold locally around $x$ and the tangent space to $Z$ at $x$ is well-defined; it is, in fact, the plane perpendicular to $\nabla{p_{sf}}(x)$ that passes through $x$.

An immediate corollary of Theorem \ref{2.2.1} is the following.

\begin{proposition}\emph{\textbf{(Guth, Katz, \cite[Proposition 3.1]{Guth_Katz_2008})}}\label{2.2.3} Let $p \in \R[x,y,z]$ be a non-zero polynomial of degree $\leq d$. Let $Z$ be the zero set of $p$. Then, $Z$ contains at most $d^2$ critical lines. \end{proposition}

\begin{proof} Since there are no squares in the expansion of $p_{sf}$ as a product of irreducible polynomials in $\R[x,y,z]$, it follows that $p_{sf}$ and $\nabla{p_{sf}}$ have no common factor. In other words, if $p_{sf}=p_1\cdots p_k$, where, for all $i \in \{1,\ldots,k\}$, $p_i$ is an irreducible polynomial in $\R[x,y,z]$, then, for all $i\in\{1,\ldots,k\}$, there exists some $g_i \in \Big\{\frac{\partial p_{sf}}{\partial x},\frac{\partial p_{sf}}{\partial y},\frac{\partial p_{sf}}{\partial z}\Big\}$, such that $p_i$ is not a factor of $g_i$. 

Now, let $l$ be a critical line of $Z$. It follows that $l$ lies in the zero set of $p_{sf}$, and therefore in the union of the zero sets of $p_1, \ldots, p_k \in \R[x,y,z]$; so, there exists $j \in \{1,\ldots,k\}$, such that $l$ lies in the zero set of $p_j$. However, since $l$ is a critical line of $Z$, it is also contained in the zero set of $\nabla{p_{sf}}$, and thus in the zero set of $g_j$ as well. Therefore, $l$ lies simultaneously in the zero sets of the polynomials $p_j$ and $g_j\in \R[x,y,z]$.

It follows from the above that the number of critical lines of $Z$ is equal to at most $\sum_{i=1,\ldots,k}L_i$, where, for all $i \in \{1,\ldots,k\}$, $L_i$ is the number of lines simultaneously contained in the zero set of $p_i$ and the zero set of $g_i$ in $\R^3$. And since the polynomials $p_i$ and $g_i \in \R[x,y,z]$ do not have a common factor, Theorem \ref{2.2.1} implies that $L_i\leq \deg p_i \cdot \deg g_i \leq \deg p_i \cdot d$, for all $i \in \{1,\ldots,k\}$. Thus, the number of critical lines of $Z$ is $\leq \sum_{i=1,\ldots,k}\deg p_i \cdot d\leq \deg p \cdot d=d^2$.

\end{proof}

\section{Proof} \label{section4}

We will now prove Theorem \ref{multsimple}, and Theorem \ref{theoremmult2} will immediately follow (as we have already mentioned, it is an application of Theorem \ref{multsimple} for $(N_1,N_2,N_3)=(1,1,1)$).

\textbf{Theorem \ref{multsimple}.} \textit{Let $\mathfrak{L}_1$, $\mathfrak{L}_2$, $\mathfrak{L}_3$ be finite collections of $L_1$, $L_2$ and $L_3$, respectively, lines in $\R^3$. Then,}
\begin{displaymath} |J_{N_1,N_2,N_3}(\mathfrak{L}_1,\mathfrak{L}_2,\mathfrak{L}_3)|\leq c \cdot \frac{(L_1L_2L_3)^{1/2}}{(N_1N_2N_3)^{1/2}}, \; \forall\;(N_1,N_2,N_3) \in \N_{+}^3, \end{displaymath}
\textit{where $c$ is a constant independent of $\mathfrak{L}_1$, $\mathfrak{L}_2$ and $\mathfrak{L}_3$.}

\begin{proof} The proof will be achieved by induction on $L_1$, $L_2$ and $L_3$. Indeed, fix $(M_1,M_2,M_3) \in \N_+^3$. For a constant $c \geq 1$ that will be specified later:

(i) For $L_1=L_2=L_3=1$,\begin{displaymath} |J_{N_1,N_2,N_3}(\mathfrak{L}_1, \mathfrak{L}_2, \mathfrak{L}_3)|  \leq c \cdot \frac{(L_1L_2L_3)^{1/2}}{(N_1N_2N_3)^{1/2}}, \; \forall\;(N_1,N_2,N_3) \in \R_{+}^3,\end{displaymath} for any collections $\mathfrak{L}_1$, $\mathfrak{L}_2$, $\mathfrak{L}_3$ in $\R^3$, each consisting of 1 line. 

Clearly, this is true for any constant $c \geq 1$.

(ii) Suppose that \begin{displaymath} |J_{N_1,N_2,N_3}(\mathfrak{L}_1, \mathfrak{L}_2, \mathfrak{L}_3)|  \leq c \cdot \frac{(L_1L_2L_3)^{1/2}}{(N_1N_2N_3)^{1/2}}, \; \forall\;(N_1,N_2,N_3) \in \R_{+}^3,\end{displaymath} for any finite collections $\mathfrak{L}_1$, $\mathfrak{L}_2$, $\mathfrak{L}_3$ of $L_1$, $L_2$ and $L_3$, respectively, lines in $\R^3$, such that $L_1\lneq M_1$, $L_2 \lneq M_2$ and $L_3 \lneq M_3$.

(iii) We will prove that  \begin{displaymath} |J_{N_1,N_2,N_3}(\mathfrak{L}_1, \mathfrak{L}_2, \mathfrak{L}_3)|  \leq c \cdot \frac{(L_1L_2L_3)^{1/2}}{(N_1N_2N_3)^{1/2}}, \; \forall\;(N_1,N_2,N_3) \in \R_{+}^3,\end{displaymath} for any finite collections $\mathfrak{L}_1$, $\mathfrak{L}_2$, $\mathfrak{L}_3$ of $L_1$, $L_2$ and $L_3$, respectively, lines in $\R^3$, such that $L_j= M_j$ for some $j \in \{1,2,3\}$ and $L_i \lneq M_i$, $L_k \lneq M_k$ for $\{i,k\}=\{1,2,3\}\setminus \{j\}$.

Indeed, fix such collections $\mathfrak{L}_1$, $\mathfrak{L}_2$ and $\mathfrak{L}_3$ of lines and let $(N_1,N_2,N_3) \in \R_{+}^3$.

For simplicity, let $$\mathfrak{G}:= J_{N_1,N_2,N_3}(\mathfrak{L}_1, \mathfrak{L}_2, \mathfrak{L}_3)$$and$$S:=|J_{N_1,N_2,N_3}(\mathfrak{L}_1, \mathfrak{L}_2, \mathfrak{L}_3)|.$$

We assume that \begin{displaymath} \frac{L_1}{\lceil N_1 \rceil} \leq \frac{L_2}{\lceil N_2 \rceil} \leq \frac{L_3}{\lceil N_3 \rceil}. \end{displaymath}

By the definition of $\mathfrak{G}$, if $x \in \mathfrak{G}$, then there exist at least $\lceil N_1 \rceil$ lines of $\mathfrak{L}_1$ and at least $\lceil N_2 \rceil$ lines of $\mathfrak{L}_2$ passing through $x$. Thus, the quantity $S\lceil N_1 \rceil \lceil N_2 \rceil$ is equal to at most the number of pairs of the form $(l_1,l_2)$, where $l_1 \in \mathfrak{L}_1$, $l_2 \in \mathfrak{L}_2$ and the lines $l_1$ and $l_2$ pass through the same point of $\mathfrak{G}$. Therefore, $S\lceil N_1 \rceil \lceil N_2 \rceil$ is equal to at most the number of all the pairs of the form $(l_1,l_2)$, where $l_1 \in \mathfrak{L}_1$ and $l_2 \in \mathfrak{L}_2$, i.e. to at most $L_1L_2$. So,
\begin{displaymath}S\lceil N_1\rceil \lceil N_2 \rceil \leq L_1L_2,\end{displaymath}
and therefore
\begin{displaymath} \frac{L_1L_2}{S\lceil N_1\rceil\lceil N_2\rceil}\geq 1. \end{displaymath}

Thus, $d:=A\frac{L_1L_2}{S\lceil N_1\rceil\lceil N_2\rceil}$ is a quantity $>1$ for $A>1$. We therefore assume that $A>1$, and we will specify its value later. Now, applying the Guth-Katz polynomial method for this $d>1$ and the finite set of points $\mathfrak{G}$, we deduce that there exists a non-zero polynomial $p\in \R[x,y,z]$, of degree $\leq d$, whose zero set $Z$ decomposes $\R^3$ in $\lesssim d^3$ cells, each of which contains $\lesssim Sd^{-3}$ points of $\mathfrak{G}$. We can assume that this polynomial is square-free, as eliminating the squares of $p$ does not inflict any change on its zero set. 

Let us assume that there are $\geq 10^{-8}S$ points of $\mathfrak{G}$ in the union of the interiors of the cells; by choosing $A$ to be a sufficiently large constant, we will be led to a contradiction.

Indeed, there are $\gtrsim S$ points of $\mathfrak{G}$ in the union of the interiors of the cells. However, there exist $\lesssim d^3$ cells in total, each containing $\lesssim Sd^{-3}$ points of $\mathfrak{G}$. Therefore, there exist $\gtrsim d^3$ cells, with $\gtrsim Sd^{-3}$ points of $\mathfrak{G}$ in the interior of each. We call the cells with this property ``full cells". 

For every full cell,  let $\mathfrak{G}_{cell}$ be the set of points of $\mathfrak{G}$ in the interior of the cell, $\mathfrak{L}_{1,cell}$ and $\mathfrak{L}_{2,cell}$ the sets of lines of $\mathfrak{L}_1$ and $\mathfrak{L}_2$, respectively, each containing at least one point of $\mathfrak{G}_{cell}$, $S_{cell}:=|\mathfrak{G}_{cell}|$, $L_{1,cell}:=|\mathfrak{L}_{1,cell}|$ and $L_{2,cell}:= |\mathfrak{L}_{2,cell}|$. Now,
\begin{displaymath}S_{cell}\lceil N_1\rceil \lceil N_2 \rceil \lesssim L_{1,cell}L_{2,cell}, \end{displaymath}
as the quantity $S_{cell}\lceil N_1\rceil \lceil N_2 \rceil$ is equal to at most the number of pairs of the form $(l_1,l_2)$, where $l_1 \in \mathfrak{L}_{1,cell}$, $l_2 \in \mathfrak{L}_{2,cell}$ and the lines $l_1$ and $l_2$ pass through the same point of $\mathfrak{G}_{cell}$. Thus, it is equal to at most $L_{1,cell}L_{2,cell}$, which is the number of all the pairs of the form $(l_1,l_2)$, where $l_1 \in \mathfrak{L}_{1,cell}$ and $l_2 \in \mathfrak{L}_{2,cell}$.

Therefore,
\begin{displaymath}(L_{1,cell}L_{2,cell})^{1/2} \gtrsim S_{cell}^{1/2} (\lceil N_1 \rceil \lceil N_2 \rceil )^{1/2} \gtrsim \frac{S^{1/2}}{d^{3/2}}(\lceil N_1\rceil\lceil N_2\rceil)^{1/2}. \end{displaymath}
But, for $i=1,2$, each of the lines of $\mathfrak{L}_{i,cell}$ intersects the boundary of the cell at at least one point $x$, with the property that the induced topology from $\R^3$ to the intersection of the line with the closure of the cell contains an open neighbourhood of $x$; therefore, there are $ \gtrsim L_{i,cell}$ incidences of this form between $\mathfrak{L}_{i,cell}$ and the boundary of the cell. Also, the union of the boundaries of all the cells is the zero set $Z$ of $p$, and if $x$ is a point of $Z$ which belongs to a line intersecting the interior of a cell, such that the induced topology from $\R^3$ to the intersection of the line with the closure of the cell contains an open neighbourhood of $x$,  then there exists at most one other cell whose interior is also intersected by the line and whose boundary contains $x$, such that the induced topology from $\R^3$ to the intersection of the line with the closure of that cell contains an open neighbourhood of $x$. So, if $I_i$ is the number of incidences between $Z$ and the lines of $\mathfrak{L}_{i}$ not lying in $Z$, $I_{i,cell}$ is the number of incidences between $\mathfrak{L}_{i,cell}$ and the boundary of the cell, and $\mathcal{C}$ is the set of all the full cells (which, in our case, has cardinality $\gtrsim d^3$), then the above imply that
\begin{displaymath} I_i \gtrsim \sum_{cell \in \mathcal{C}} I_{i,cell} \gtrsim \sum_{cell \in \mathcal{C}}L_{i,cell},\text{ for }i=1,2.\end{displaymath}
On the other hand, if a line does not lie in the zero set $Z$ of $p$, then it intersects $Z$ in $\leq d$ points. Thus,
\begin{displaymath} I_i \leq L_{i} \cdot d,\text{ for }i=1,2. \end{displaymath}
Therefore, 
\begin{displaymath} \sum_{cell \in \mathcal{C}} \frac{S^{1/2}}{d^{3/2}}(\lceil N_1\rceil\lceil N_2 \rceil)^{1/2} \lesssim \sum_{cell \in \mathcal{C}} (L_{1,cell}L_{2,cell})^{1/2} \lesssim \end{displaymath}
\begin{displaymath} \lesssim\Bigg(\sum_{cell \in\mathcal{C}} L_{1,cell}\Bigg)^{1/2} \Bigg(\sum_{cell \in\mathcal{C}}L_{2,cell}\Bigg)^{1/2}\lesssim  I_1^{1/2}I_2^{1/2}\lesssim (L_1d)^{1/2}(L_2d)^{1/2} \sim\end{displaymath}
\begin{displaymath} \sim (L_1L_2)^{1/2}d.\end{displaymath}

But the full cells number $\gtrsim d^3$. Thus, 
\begin{displaymath} S^{1/2}(\lceil N_1\rceil\lceil N_2\rceil)^{1/2}d^{3/2} \lesssim (L_1L_2)^{1/2}d,\end{displaymath}
from which we obtain
\begin{displaymath} d \lesssim \frac{L_1L_2}{S\lceil N_1 \rceil\lceil N_2\rceil}, \end{displaymath}
which in turn gives $A\lesssim 1$. In other words, there exists some explicit constant $C$, such that $A \leq C$. By fixing $A$ to be a number larger than $C$ (and of course larger than 1, to have that $d>1$), we are led to a contradiction.

Therefore, there are $< 10^{-8}S$ points of $\mathfrak{G}$ in the union of the interiors of the cells. Thus, if $\mathfrak{G}_1$ denotes the set of points in $\mathfrak{G}$ which lie in $Z$, it holds that $|\mathfrak{G}_1|>(1-10^{-8})S$. 

Now, by the definition of $\mathfrak{G}$, for each $x \in \mathfrak{G}$ we can fix collections $\mathfrak{L}_1(x) \subseteq \mathfrak{L}_1$, $\mathfrak{L}_2(x) \subseteq \mathfrak{L}_2$ and $\mathfrak{L}_3(x) \subseteq \mathfrak{L}_3$ of lines passing through $x$, such that $|\mathfrak{L}_1(x)|= \lceil N_1 \rceil$, $|\mathfrak{L}_2(x)|= \lceil N_2 \rceil$ and $|\mathfrak{L}_3(x)|= \lceil N_3 \rceil$, and, if $l_1 \in \mathfrak{L}_1(x)$, $l_2 \in \mathfrak{L}_2(x)$ and $l_3 \in \mathfrak{L}_3(x)$, then the directions of the lines $l_1$, $l_2$ and $l_3$ span $\R^3$.

We now continue by refining the sets of lines and points we are interested in, based on the incidences between each $x \in \mathfrak{G}_1$ and the lines only in $\mathfrak{L}_1(x)$, $\mathfrak{L}_2(x)$ and $\mathfrak{L}_3(x)$, rather than all the lines in $\mathfrak{L}_1$, $\mathfrak{L}_2$ and $\mathfrak{L}_3$ passing through $x$.

In particular, for all $j \in \{1,2,3\}$, we can define $\mathfrak{L}_j':=\bigg\{l \in \mathfrak{L}_j:\Big |\Big\{x \in \mathfrak{G}_1: l \in \mathfrak{L}_j(x)\Big |\geq \frac{1}{100} \frac{S\lceil N_j\rceil}{L_j}\Big\}\bigg\}$. In other words, for all $j \in \{1,2,3\}$, $\mathfrak{L}_j'$ is the set of lines in $\mathfrak{L}_j$, each of which belongs to $\mathfrak{L}_j(x)$ for at least $\frac{1}{100} \frac{S\lceil N_j \rceil }{L_j}$ points $x \in \mathfrak{G}_1$.

Moreover, for all $j \in \{1,2,3\}$, for any subset $\mathcal{G}$ of $\mathfrak{G}$ and any subset $\mathcal{L}$ of $\mathfrak{L}_j$, we denote by $I^{(j)}_{\mathcal{G},\mathcal{L}}$ the number of pairs of the form $(x, l)$, where $x \in \mathcal{G}$ and $l \in \mathfrak{L}_j(x) \cap\mathcal{L}$; note that the set of these pairs is a subset of the set of incidences between $\mathcal{G}$ and $\mathcal{L}$.

We now analyse the situation.

Let $j\in \{1,2,3\}$. Each point $x \in \mathfrak{G}_1$ intersects $\lceil N_j \rceil$ lines of $\mathfrak{L}_j(x)$, which is a subset of $\mathfrak{L}_j$. Thus,  \begin{displaymath} I^{(j)}_{\mathfrak{G}_1,\mathfrak{L}_j} > (1-10^{-8})S\lceil N_j \rceil.\end{displaymath}
On the other hand, each line $l \in \mathfrak{L}_j \setminus \mathfrak{L}_j'$ belongs to $\mathfrak{L}_j(x)$ for fewer than $\frac{1}{100} \frac{S\lceil N_j\rceil}{L_j}$ points $x \in \mathfrak{G}_1$, so
\begin{displaymath}I^{(j)}_{\mathfrak{G}_1,\mathfrak{L}_j\setminus \mathfrak{L}'_j} \leq  |\mathfrak{L}_j\setminus \mathfrak{L}'_j| \cdot \frac{S\lceil N_j \rceil}{100L_j} \leq \frac{1}{100}S\lceil N_j\rceil.\end{displaymath}
Therefore, since $ I^{(j)}_{\mathfrak{G}_1, \mathfrak{L}_j}=I^{(j)}_{\mathfrak{G}_1,\mathfrak{L}_j \setminus \mathfrak{L}'_j}+I^{(j)}_{\mathfrak{G}_1,\mathfrak{L}'_j}$, it follows that 
\begin{displaymath}I^{(j)}_{\mathfrak{G}_1,\mathfrak{L}'_j}> (1-10^{-8}-10^{-2})S\lceil N_j \rceil, \end{displaymath}
for all $j \in \{1,2,3\}$.

Now, for all $j \in \{1,2,3\}$, we define $\mathfrak{G}'_j:=\Big\{x \in \mathfrak{G}_1: |\mathfrak{L}_j(x) \cap \mathfrak{L}_j'|\geq \frac{10^{-8}}{1+10^{-8}}(1-10^{-8}-10^{-2})\lceil N_j \rceil\Big\}$. In other words, for all $j \in \{1,2,3\}$, $x \in \mathfrak{G}_j'$ if and only if $x \in \mathfrak{G}_1$ and $\mathfrak{L}_j(x)$ contains at least $\frac{10^{-8}}{1+10^{-8}}(1-10^{-8}-10^{-2})\lceil N_j \rceil$ lines of $\mathfrak{L}'_j$.

Let $j \in \{1,2,3\}$. Since each point $x \in \mathfrak{G}_1 \setminus \mathfrak{G}_j'$ intersects fewer than $\frac{10^{-8}}{1+10^{-8}}(1-10^{-8}-10^{-2})\lceil N_j \rceil$ lines of $\mathfrak{L}_j(x) \cap \mathfrak{L}'_j$, it follows that \begin{displaymath}I^{(j)}_{\mathfrak{G}_1 \setminus \mathfrak{G}'_j, \mathfrak{L}'_j} < |\mathfrak{G}_1 \setminus \mathfrak{G}'_j|\frac{10^{-8}}{1+10^{-8}}(1-10^{-8}-10^{-2})\lceil N_j \rceil\leq  \end{displaymath}
\begin{displaymath} \leq \frac{10^{-8}}{1+10^{-8}}(1-10^{-8}-10^{-2})S\lceil N_j \rceil.\end{displaymath}
Therefore, since $I^{(j)}_{\mathfrak{G}_1, \mathfrak{L}'_j}=I^{(j)}_{\mathfrak{G}_1 \setminus \mathfrak{G}'_j, \mathfrak{L}'_j}+I^{(j)}_{\mathfrak{G}'_j, \mathfrak{L}'_j}$, we obtain
\begin{displaymath}I^{(j)}_{\mathfrak{G}'_j, \mathfrak{L}'_j} > \frac{1-10^{-8}-10^{-2}}{1+10^{-8}}S\lceil N_j \rceil. \end{displaymath}
At the same time, however, $|\mathfrak{L}_j(x)| =\lceil N_j \rceil$ for all $x \in \mathfrak{G}_j'$, and thus  $I^{(j)}_{\mathfrak{G}'_j, \mathfrak{L}'_j} \leq |\mathfrak{G}'_j|\lceil N_j\rceil$. Therefore, $ \frac{1-10^{-8}-10^{-2}}{1+10^{-8}}S\lceil N_j \rceil< |\mathfrak{G}'_j|\lceil N_j\rceil$, which implies that
\begin{displaymath} |\mathfrak{G}'_j|\geq \frac{1-10^{-8}-10^{-2}}{1+10^{-8}}S, \end{displaymath}
for all $j \in \{1,2,3\}$. In other words, for all $j\in \{1,2,3\}$, there exist at least $\frac{1-10^{-8}-10^{-2}}{1+10^{-8}}S$ points $x \in \mathfrak{G}_1$ such that $x$ intersects at least $\frac{10^{-8}}{1+10^{-8}}(1-10^{-8}-10^{-2})\lceil N_j\rceil$ lines of $\mathfrak{L}_j(x) \cap \mathfrak{L}'_j$.

But \begin{displaymath} |\mathfrak{G}_1'\cup \mathfrak{G}_2' \cup \mathfrak{G}_3'|=|\mathfrak{G}_1'|+|\mathfrak{G}_2'|+|\mathfrak{G}_3'|-|\mathfrak{G}_1'\cap\mathfrak{G}_2'|-|\mathfrak{G}_2'\cap\mathfrak{G}_3'|-|\mathfrak{G}_1'\cap\mathfrak{G}_3'|+|\mathfrak{G}_1'\cap\mathfrak{G}_2'\cap\mathfrak{G}_3'|,\end{displaymath} and thus 
$$|\mathfrak{G}_1'\cap\mathfrak{G}_2'\cap\mathfrak{G}_3'|=$$
$$=|\mathfrak{G}_1'\cup \mathfrak{G}_2' \cup \mathfrak{G}_3'|-(|\mathfrak{G}_1'|+|\mathfrak{G}_2'|+|\mathfrak{G}_3'|)+(|\mathfrak{G}_1'\cap\mathfrak{G}_2'|+|\mathfrak{G}_2'\cap\mathfrak{G}_3'|+|\mathfrak{G}_1'\cap\mathfrak{G}_3'|)\geq
$$
$$\geq |\mathfrak{G}_1'|-(|\mathfrak{G}_1'|+|\mathfrak{G}_2'|+|\mathfrak{G}_3'|)+$$
$$+\Big((|\mathfrak{G}_1'|+|\mathfrak{G}_2'|-|\mathfrak{G}_1'\cup\mathfrak{G}_2'|)+(|\mathfrak{G}_2'|+|\mathfrak{G}_3'|-|\mathfrak{G}_2'\cup\mathfrak{G}_3'|)+(|\mathfrak{G}_1'|+|\mathfrak{G}_3'|-|\mathfrak{G}_1'\cup\mathfrak{G}_3'|)\Big)\geq
$$
$$\geq2|\mathfrak{G}_1'|+|\mathfrak{G}_2'|+|\mathfrak{G}_3'|-|\mathfrak{G}_1'\cup\mathfrak{G}_2'|-|\mathfrak{G}_2'\cup\mathfrak{G}_3'|-|\mathfrak{G}_1'\cup\mathfrak{G}_3'|\geq
$$
\begin{displaymath}\geq4\cdot\frac{1-10^{-8}-10^{-2}}{1+10^{-8}}S-3S=\end{displaymath} 
\begin{displaymath}= \frac{4(1-10^{-8}-10^{-2})-3(1+10^{-8})}{1+10^{-8}}S=\frac{1-7 \cdot 10^{-8}-4 \cdot 10^{-2}}{1+10^{-8}}S\geq 
\end{displaymath}
\begin{displaymath} \geq \frac{1-8\cdot 10^{-2}}{1+10^{-8}}S;\end{displaymath}
in other words, there exist at least $\frac{1-8\cdot 10^{-2}}{1+10^{-8}}S$ points $x \in \mathfrak{G}_1$ intersecting at least $\frac{10^{-8}}{1+10^{-8}}(1-10^{-8}-10^{-2})\lceil N_j\rceil$ lines of $\mathfrak{L}_j(x) \cap \mathfrak{L}'_j$, simultaneously for all $j\in\{1,2,3\}$.

\textbf{Case 1:} Suppose that, for some $j\in \{1,2,3\}$, it holds that $\frac{1}{10^{100}} \frac{S\lceil N_j\rceil}{L_j}\leq d$. Then $\frac{S\lceil N_j\rceil}{L_j} \lesssim \frac{L_1L_2}{S\lceil N_1\rceil\lceil N_2\rceil}$, which implies that
\begin{displaymath} S\lesssim \Bigg(\frac{L_1L_2}{\lceil N_1\rceil \lceil N_2\rceil}\Bigg)^{1/2}\Bigg(\frac{L_j}{\lceil N_j\rceil}\Bigg)^{1/2} \lesssim \frac{(L_1L_2L_3)^{1/2}}{(\lceil N_1\rceil \lceil N_2\rceil \lceil N_3\rceil )^{1/2}}\lesssim \frac{(L_1L_2L_3)^{1/2}}{(N_1N_2N_3)^{1/2}}. \end{displaymath}

\textbf{Case 2:} Suppose that $\frac{1}{10^{100}} \frac{S\lceil N_j\rceil}{L_j}> d$, for all $j=1,2,3$. Then, each line in $\mathfrak{L}_1'$, $\mathfrak{L}_2'$ and $\mathfrak{L}_3'$ lies in $Z$, therefore each point in $\mathfrak{G}_1'\cap \mathfrak{G}_2' \cap \mathfrak{G}_3'$ is a critical point of $Z$. 

Now, for all $j \in \{1,2,3\}$, we define $\mathfrak{L}_{j,1}:=\big\{l \in \mathfrak{L}_j: |\{x \in \mathfrak{G}_1' \cap \mathfrak{G}_2' \cap \mathfrak{G}_3': l \in \mathfrak{L}_j(x)\}|\geq\frac{1}{10^{100}}S\lceil N_{j}\rceil L_{j}^{-1}\big\}$. In other words, for all $j \in \{1,2,3\}$, $\mathfrak{L}_{j,1}$ is the set of lines in $\mathfrak{L}_{j}$, each of which belongs to $\mathfrak{L}_j(x)$ for at least $\frac{1}{10^{100}}S\lceil N_{j}\rceil L_{j}^{-1}$ points $x \in \mathfrak{G}_1' \cap \mathfrak{G}_2' \cap \mathfrak{G}_3'$. 

$\bullet$ Suppose that, for some $j \in \{1,2,3\}$, $|\mathfrak{L}_{j,1}|\geq \frac{L_j}{10^{1000}}$. Each line in $\mathfrak{L}_{j,1}$ contains more than $d$ critical points of $Z$, it is therefore a critical line. Thus, \begin{displaymath}\frac{L_j}{10^{1000}} \leq d^2, \end{displaymath}so \begin{displaymath} L_j \lesssim \frac{(L_1L_2)^2}{S^2(\lceil N_1\rceil \lceil N_2\rceil)^{2}}, \end{displaymath} from which it follows that 
\begin{displaymath}S\lesssim \frac{L_1L_2}{\lceil N_1\rceil \lceil N_2\rceil}\frac{1}{L_j^{1/2}} \lesssim \frac{(L_1L_2L_3)^{1/2}}{(\lceil N_1\rceil \lceil N_2\rceil \lceil N_3\rceil )^{1/2}}\lesssim \frac{(L_1L_2L_3)^{1/2}}{(N_1N_2N_3)^{1/2}}.\end{displaymath}
We are now ready to define the constant $c$ appearing in our induction process. Indeed, there exists some constant $c' \geq 1$, independent of $\mathfrak{L}_1$, $\mathfrak{L}_2$, $\mathfrak{L}_3$ and $N_1$, $N_2$ and $N_3$, such that \begin{displaymath} S \leq c' \cdot \frac{(L_1L_2L_3)^{1/2}}{(N_1N_2N_3)^{1/2}} \end{displaymath} in all the cases dealt with so far. Let $c$ be such a constant $c'$.

$\bullet$ Suppose that, for all $i \in \{1,2,3\}$, $|\mathfrak{L}_{i,1}|< \frac{L_i}{10^{1000}}$. Then, it holds that $|\mathfrak{L}_{j,1}|< \frac{L_j}{10^{1000}}$ in particular for that $j \in \{1,2,3\}$ such that $L_j=M_j$; we now fix that $j \in \{1,2,3\}$.

Independently of the fact that $|\mathfrak{L}_{j,1}|< \frac{L_j}{10^{1000}}$, it holds that
\begin{displaymath} I^{(j)}_{\mathfrak{G}_1'\cap\mathfrak{G}_2'\cap\mathfrak{G}_3', \mathfrak{L}_j'} \geq |\mathfrak{G}_1'\cap\mathfrak{G}_2'\cap\mathfrak{G}_3'| \cdot \frac{10^{-8}}{1+10^{-8}}(1-10^{-8}-10^{-2})\lceil N_j \rceil\geq \end{displaymath}
\begin{displaymath} \geq \frac{1-8\cdot 10^{-2}}{1+10^{-8}}S \cdot \frac{10^{-8}}{1+10^{-8}}(1-10^{-8}-10^{-2})\lceil N_j\rceil \geq 10^{-10}S\lceil N_j\rceil,\end{displaymath}
since each point $x \in \mathfrak{G}_1'\cap\mathfrak{G}_2'\cap\mathfrak{G}_3'$ intersects at least $\frac{10^{-8}}{1+10^{-8}}(1-10^{-8}-10^{-2})\lceil N_j \rceil $ lines of $\mathfrak{L}_j(x) \cap \mathfrak{L}'_j$.

In addition, each line $l \in \mathfrak{L}_j' \setminus \mathfrak{L}_{j,1}$ belongs to $\mathfrak{L}_j(x)$ for fewer than $\frac{1}{10^{100}}\frac{S\lceil N_j\rceil}{L_j}$ points $x \in \mathfrak{G}_1'\cap\mathfrak{G}_2'\cap\mathfrak{G}_3'$. Thus,
\begin{displaymath} I^{(j)}_{\mathfrak{G}_1'\cap\mathfrak{G}_2'\cap\mathfrak{G}_3', \mathfrak{L}_j'\setminus \mathfrak{L}_{j,1}} <|\mathfrak{L}_j' \setminus \mathfrak{L}_{j,1}| \cdot \frac{S\lceil N_j\rceil }{10^{100}L_j} \leq L_j \cdot \frac{S\lceil N_j\rceil}{10^{100}L_j}=10^{-100}S\lceil N_j\rceil.\end{displaymath} Therefore, since $I^{(j)}_{\mathfrak{G}_1'\cap\mathfrak{G}_2'\cap\mathfrak{G}_3', \mathfrak{L}_j'}=I^{(j)}_{\mathfrak{G}_1'\cap\mathfrak{G}_2'\cap\mathfrak{G}_3', \mathfrak{L}_j'\setminus \mathfrak{L}_{j,1}}+I^{(j)}_{\mathfrak{G}_1'\cap\mathfrak{G}_2'\cap\mathfrak{G}_3',  \mathfrak{L}_{j,1}}$, we obtain
\begin{displaymath} I^{(j)}_{\mathfrak{G}_1'\cap\mathfrak{G}_2'\cap\mathfrak{G}_3',  \mathfrak{L}_{j,1}} >10^{-11}S\lceil N_j\rceil . \end{displaymath} 
Now, again for that $j \in \{1,2,3\}$ such that $L_j=M_j$, we define $\mathfrak{G}':=\{x \in \mathfrak{G}_1'\cap\mathfrak{G}_2'\cap\mathfrak{G}_3': |\mathfrak{L}_j(x) \cap \mathfrak{L}_{j,1}| \geq 10^{-12}\lceil N_j\rceil\}$. In other words, for that particular $j \in \{1,2,3\}$, $x \in \mathfrak{G}'$ if and only if $x \in \mathfrak{G}_1'\cap\mathfrak{G}_2'\cap\mathfrak{G}_3'$ and $\mathfrak{L}_j(x)$ contains at least $10^{-12}\lceil N_j\rceil$ lines of $\mathfrak{L}_{j,1}$. 

Since each point $x \in (\mathfrak{G}_1'\cap\mathfrak{G}_2'\cap\mathfrak{G}_3')\setminus \mathfrak{G}'$ intersects fewer than $10^{-12}\lceil N_j\rceil$ lines of $\mathfrak{L}_j(x) \cap \mathfrak{L}_{j,1}$, it holds that
\begin{displaymath} I^{(j)}_{(\mathfrak{G}_1'\cap\mathfrak{G}_2'\cap\mathfrak{G}_3')\setminus \mathfrak{G}',  \mathfrak{L}_{j,1}} <10^{-12}S\lceil N_j\rceil, \end{displaymath} and thus, as $I^{(j)}_{\mathfrak{G}_1'\cap\mathfrak{G}_2'\cap\mathfrak{G}_3',  \mathfrak{L}_{j,1}}=I^{(j)}_{(\mathfrak{G}_1'\cap\mathfrak{G}_2'\cap\mathfrak{G}_3')\setminus \mathfrak{G}',  \mathfrak{L}_{j,1}}+I^{(j)}_{\mathfrak{G}', \mathfrak{L}_{j,1}}$, we obtain
\begin{displaymath} I^{(j)}_{\mathfrak{G}', \mathfrak{L}_{j,1}}> (10^{-11}-10^{-12})S\lceil N_j\rceil = 9 \cdot 10^{-12}S\lceil N_j\rceil. \end{displaymath} At the same time, however, $|\mathfrak{L}_j(x)| =\lceil N_j\rceil$ for all $x \in \mathfrak{G}'$. Therefore, \begin{displaymath} I^{(j)}_{\mathfrak{G}', \mathfrak{L}_{j,1}} \leq |\mathfrak{G}'|\lceil N_j\rceil. \end{displaymath} Thus, the above imply that
\begin{displaymath} |\mathfrak{G}'| > 9\cdot 10^{-12}S. \end{displaymath}

But if $\{i,k\}=\{1,2,3\}\setminus \{j\}$, then each point $x \in \mathfrak{G}'$ is a multijoint for the collections $\mathfrak{L}_{j,1}$, $\mathfrak{L}_i'$ and $\mathfrak{L}_k'$ of lines, that lies in $\geq 10^{-12}\lceil N_j\rceil$ lines of $\mathfrak{L}_j(x) \cap \mathfrak{L}_{j,1}$, in $\geq  \frac{1-10^{-8}-10^{-2}}{1+10^{8}}\lceil N_i \rceil $ lines of $\mathfrak{L}_i(x)\cap \mathfrak{L}_i'$ and in $\geq  \frac{1-10^{-8}-10^{-2}}{1+10^{8}}\lceil N_k\rceil $ lines of $\mathfrak{L}_k(x) \cap \mathfrak{L}_k'$. Now, for all $x \in \mathfrak{G}'$, if $l_j \in \mathfrak{L}_j(x)$, $l_i \in \mathfrak{L}_i(x)$ and $l_k \in \mathfrak{L}_k(x)$, then the directions of the lines $l_i$, $l_j$ and $l_k$ span $\R^3$. Therefore, since $|\mathfrak{L}_{j,1}| < \frac{{L}_j}{10^{1000}}\lneq M_j$, our induction hypothesis implies that
\begin{displaymath} 9\cdot 10^{-12}S< |\mathfrak{G}'| \leq 
\end{displaymath}
\begin{displaymath}
\leq c\cdot \frac{(|\mathfrak{L}_{j,1}| \cdot |\mathfrak{L}_i'| \cdot |\mathfrak{L}_k'|)^{1/2}}{(10^{-12}\lceil N_j\rceil )^{1/2} \Big(\frac{1-10^{-8}-10^{-2}}{1+10^{8}}\lceil N_i\rceil \Big)^{1/2}\Big(\frac{1-10^{-8}-10^{-2}}{1+10^{8}}\lceil N_k\rceil \Big)^{1/2}},\end{displaymath} 
where $c$ is the explicit constant defined earlier, and which appears in the induction process.

Therefore, $$S\leq c\cdot \frac{1}{9\cdot 10^{-12} (10^{-12})^{1/2}\Big(\frac{1-10^{-8}-10^{-2}}{1+10^{8}}\Big)^{1/2}\Big(\frac{1-10^{-8}-10^{-2}}{1+10^{8}}\Big)^{1/2}}\cdot$$
$$\cdot
 \frac{(|\mathfrak{L}_{j,1}| \cdot |\mathfrak{L}_i'| \cdot |\mathfrak{L}_k'|)^{1/2}}{(\lceil N_j\rceil \lceil N_i\rceil \lceil N_k\rceil )^{1/2}}\leq$$
$$\leq c\cdot \frac{1}{9\cdot 10^{-12} (10^{-12})^{1/2}\Big(\frac{1-10^{-8}-10^{-2}}{1+10^{8}}\Big)^{1/2}\Big(\frac{1-10^{-8}-10^{-2}}{1+10^{8}}\Big)^{1/2}}\cdot $$
$$ \cdot
 \frac{(|\mathfrak{L}_{j,1}| \cdot |\mathfrak{L}_i'| \cdot |\mathfrak{L}_k'|)^{1/2}}{( N_j N_i N_k)^{1/2}}.
$$
However, 
$$\frac{1}{9\cdot 10^{-12} (10^{-12})^{1/2}\Big(\frac{1-10^{-8}-10^{-2}}{1+10^{8}}\Big)^{1/2}\Big(\frac{1-10^{-8}-10^{-2}}{1+10^{8}}\Big)^{1/2}}<
$$
$$<\frac{1}{10^{-18} \cdot \frac{1-10^{-8}-10^{-2}}{1+10^{8}}}<\frac{1}{10^{-18} \cdot \frac{1/2}{2 \cdot 10^{8}}}=4 \cdot \frac{10^8}{10^{-18}}<10^{27},
$$
so, since \begin{displaymath} |\mathfrak{L}_{j,1}|< \frac{L_j}{10^{1000}}, \text{ }|\mathfrak{L}_i'|\leq L_i\text{ and }|\mathfrak{L}_k'| \leq L_k, \end{displaymath} it follows that
\begin{displaymath} S \leq c \cdot \frac{(L_1L_2L_3)^{1/2}}{(N_1N_2N_3)^{1/2}}, \end{displaymath}
for the same constant $c$ as in the first two steps of the induction process.

Thus, Theorem \ref{multsimple} is proved.

\end{proof}

Theorem \ref{theoremmult2} is an immediate corollary of Theorem \ref{multsimple} for $(N_1,N_2,N_3)=(1,1,1)$.

We now briefly explain why Theorem \ref{theoremmult2} is actually equivalent to Theorem \ref{multsimple}, using a standard probabilistic argument (see, for example, Guth's notes \href{http://math.mit.edu/~lguth/PolyMethod/lect34.pdf}{``The multilinear Kakeya inequality"}, or, for more details, \cite[Chapter 7]{Iliopoulou_13}). Indeed, suppose that $$ |J(\mathfrak{L}_1,\mathfrak{L}_2,\mathfrak{L}_3)| \lesssim (|\mathfrak{L}_1| |\mathfrak{L}_2||\mathfrak{L}_3|)^{1/2},$$ for all finite collections $\mathfrak{L}_1$, $\mathfrak{L}_2$, $\mathfrak{L}_3$ of lines in $\R^3$. 
Now, fix finite collections $\mathfrak{L}_1$, $\mathfrak{L}_2$, $\mathfrak{L}_3$ of $L_1$, $L_2$, $L_3$, respectively, lines in $\R^n$, and $(N_1,N_2,N_3) \in {\N}_+^3$. For all $i=1,2,3$, we create a subcollection $\mathfrak{L}_i'$ of $\mathfrak{L}_i$, by choosing each line of $\mathfrak{L}_i$ with probability $1/N_i$. The probability that a point $x \in J_{N_1,N_2,N_3}(\mathfrak{L}_1,\mathfrak{L}_2,\mathfrak{L}_3)$ belongs to at least one line of $\mathfrak{L}_i(x) \cap \mathfrak{L}_i'$ is $\gtrsim 1$. Moreover, with probability $\gtrsim 1$, $|\mathfrak{L}_i'| \lesssim L_i/N_i$, for all $i=1,2,3$. It follows that there exist finite collections $\mathfrak{L}_1'$, $\mathfrak{L}_2'$, $\mathfrak{L}_3'$ of lines in $\R^3$, with cardinalities $\lesssim L_1/N_1$, $\lesssim L_2/N_2$, $\lesssim L_3/N_3$, respectively, such that $\gtrsim |J_{N_1,N_2,N_3}(\mathfrak{L}_1,\mathfrak{L}_2,\mathfrak{L}_3)|$ multijoints in $J_{N_1,N_2,N_3}(\mathfrak{L}_1,\mathfrak{L}_2,\mathfrak{L}_3)$ belong to $J(\mathfrak{L}_1',\mathfrak{L}_2',\mathfrak{L}_3')$. This implies that $|J_{N_1,N_2,N_3}(\mathfrak{L}_1,\mathfrak{L}_2,\mathfrak{L}_3)|\lesssim|J(\mathfrak{L}_1',\mathfrak{L}_2',\mathfrak{L}_3')|\lesssim \big(\frac{L_1}{N_1}\frac{L_2}{N_2} \frac{L_3}{N_3}\big)^{1/2}\sim \frac{(L_1L_2 L_3)^{1/2}}{(N_1N_2 N_3)^{1/2}}$.

It follows that Theorem \ref{theoremmult2} is equivalent to Theorem \ref{multsimple}.

\section{From lines to curves} \label{section5}

In Section \ref{section6}, we will consider multijoints formed by more general curves, and we will extend Theorem \ref{theoremmult2} and Theorem \ref{multsimple} to that setting. To that end, we need to extend certain computational results for sets of lines to sets of more general curves, by recalling and further analysing some facts from algebraic geometry.

If $\mathbb{K}$ is a field, then any set of the form \begin{displaymath} \{x \in \mathbb{K}^n: p_i(x)=0, \; \forall \; i=1,\ldots,k\}, \end{displaymath} where $k \in \N$ and $p_i \in \mathbb{K}[x_1,\ldots,x_n]$ for all $i =1,\ldots,k$, is called an \textit{algebraic set} or an \textit{affine variety} or simply a \textit{variety} in $\mathbb{K}^n$, and is denoted by $V(p_1,\ldots,p_k)$. A variety $V$ in $\mathbb{K}^n$ is \textit{irreducible} if it cannot be expressed as the union of two non-empty varieties in $\mathbb{K}^n$ which are strict subsets of $V$.

If $V$ is a variety in $\mathbb{K}^n$, the set \begin{displaymath} I(V):= \{ p \in \mathbb{K}[x_1,\ldots,x_n]: p(x)=0, \; \forall \; x \in V\} \end{displaymath}
is an ideal in $\mathbb{K}[x_1,\ldots,x_n]$. If, in particular, $V$ is irreducible, then $I(V)$ is a prime ideal of $\mathbb{K}[x_1,\ldots,x_n]$, and the transcendence degree of the ring $\mathbb{K}[x_1,\ldots,x_n]/I(V)$ over $\mathbb{K}$ is the \textit{dimension} of the irreducible variety $V$. The \textit{dimension of an algebraic set} is the maximal dimension of all the irreducible varieties contained in the set. If an algebraic set has dimension 1 it is called an \textit{algebraic curve}, while if it has dimension $n-1$ it is called an \textit{algebraic hypersurface}.

Now, if $\mathbb{K}$ is a field, an order $\prec$ on the set of monomials in $\mathbb{K}[x_1,...,x_n]$ is called a \textit{term order}, if it is a total order on the monomials of $\mathbb{K}[x_1,...,x_n]$, such that it is multiplicative (i.e. it is preserved by multiplication by the same monomial) and the constant monomial is the $\prec$-smallest monomial. Then, if $I$ is an ideal in $\mathbb{K}[x_1,...,x_n]$, we define $in_{\prec}(I)$ as the ideal of $\mathbb{K}[x_1,...,x_n]$ generated by the $\prec$-initial terms, i.e. the $\prec$-largest monomial terms, of all the polynomials in $I$. 

Let $V$ be a variety in $\mathbb{K}[x_1,...,x_n]$ and $\prec$ a term order on the set of monomials in $\mathbb{K}[x_1,...,x_n]$. Also, let $S$ be a maximal subset of the set of variables $\{x_1,...,x_n\}$, with the property that no monomial in the variables in $S$ belongs to $in_{\prec}(I(V))$. Then, it holds that the dimension of $V$ is the cardinality of $S$ (see \cite{Sturmfels_2005}).

Now, if $\gamma$ is an algebraic curve in $\C^n$, a generic hyperplane of $\C^n$ intersects the curve in a specific number of points (counted with appropriate multiplicities), which is called the \textit{degree} of the curve.

A consequence of B\'ezout's theorem (see, for example, \cite[Theorem 12.3]{MR732620} or \cite[Chapter 3, \S3]{MR2122859}) is the following.

\begin{theorem}\emph{\textbf{(B\'ezout)}}\label{4.1.2} Let $\gamma$ be an irreducible algebraic curve in $\C^n$ of degree $b$, and $p \in \C[x_1,\ldots,x_n]$. If $\gamma$ is not contained in the zero set of $p$, it intersects the zero set of $p$ at most $b \cdot \deg p$ times. \end{theorem}

By B\'ezout's theorem, we can deduce the following (see \cite{Iliopoulou_12} for details).

\begin{corollary}\label{4.1.4} Let $\gamma _1$, $\gamma_2$ be two distinct irreducible complex algebraic curves in $\C^n$. Then, they have at most $\deg \gamma_1\cdot \deg \gamma_2$ common points. \end{corollary}

Using Corollary \ref{4.1.4}, we have proved the following in \cite{Iliopoulou_12}.

\begin{lemma}\label{4.1.1} An irreducible real algebraic curve $\gamma$ in $\R^n$ is contained in a unique irreducible complex algebraic curve in $\C^n$.
\end{lemma}

Note that, by the above, the smallest complex algebraic curve containing a real algebraic curve is the union of the irreducible complex algebraic curves, each of which contains an irreducible component of the real algebraic curve.

In particular, we have shown the following in \cite{Iliopoulou_13}.

\begin{lemma} \label{newstuff} Any real algebraic curve in $\R^n$ is the intersection of $\R^n$ with the smallest complex algebraic curve containing it. 
\end{lemma}

\begin{proof} Let $\gamma$ be a real algebraic curve in $\R^n$ and $\gamma _{\C}$ the smallest complex algebraic curve containing it. We will show that $\gamma = \R^n \cap \gamma _{\C}$. 

Let $x \in \R^n$, such that $x \notin \gamma$; then, $x \notin \gamma_{\C}$. Indeed, $\gamma$ is the intersection of the zero sets, in $\R^n$, of some polynomials $p_1, \ldots, p_k \in \R[x_1,\ldots,x_n]$. Since $x \notin \gamma$, it follows that $x$ does not belong to the zero set of $p_i$ in $\R^n$, for some $i \in \{1,\ldots,k\}$. However, $x \in \R^n$, so it does not belong to the zero set of $p_i$ in $\C^n$, either. 

Now, the zero set of $p_i$ in $\C^n$ is a complex algebraic set containing $\gamma$, and therefore its intersection with $\gamma_{\C}$ is a complex algebraic set containing $\gamma$; in fact, it is a complex algebraic curve, since it contains the infinite set $\gamma$ and lies inside the complex algebraic curve $\gamma _{\C}$. Therefore, the intersection of the zero set of $p_i$ in $\C^n$ and $\gamma_{\C}$ is equal to $\gamma _{\C}$, as otherwise it would be a complex algebraic curve, smaller that $\gamma_{\C}$, containing $\gamma$. This means that $\gamma_{\C}$ is contained in the zero set of $p_i$ in $\C^n$, and since $x$ does not belong to the zero set of $p_i$ in $\C^n$, it does not belong to $\gamma _{\C}$ either.

Therefore, $\gamma = \R^n \cap \gamma _{\C}$.

\end{proof}

Now, even though a generic hyperplane of $\C^n$ intersects a complex algebraic curve in $\C^n$ in a fixed number of points, this is not true in general for real algebraic curves. However, by Lemma \ref{4.1.1}, we can define the \textit{degree of an irreducible real algebraic curve} in $\R^n$ as the degree of the (unique) irreducible complex algebraic curve in $\C^n$ containing it. Furthermore, we can define the \textit{degree of a real algebraic curve} in $\R^n$ as the degree of the smallest complex algebraic curve in $\C^n$ containing it. With this definition, and due to Lemma \ref{newstuff}, the degree of a real algebraic curve in $\R^n$ is equal to the sum of the degrees of its irreducible components (Lemma \ref{newstuff} ensures that distinct irreducible components of a real algebraic curve in $\R^n$ are contained in distinct irreducible complex algebraic curves in $\C^n$).

Therefore, if, by saying that a real algebraic curve $\gamma$ in $\R^n$ \textit{crosses itself at the point }$x_0 \in \gamma$, we mean that any neighbourhood of $x_0$ in $\gamma$ is homeomorphic to at least two intersecting lines, it follows that a real algebraic curve in $\R^n$ crosses itself at a point at most as many times as its degree.

Moreover, an immediate consequence of the discussion above is the following.

\begin{corollary}\label{4.1.3} Let $\gamma$ be an irreducible real algebraic curve in $\R^n$ of degree $b$, and $p \in \R[x_1,\ldots,x_n]$. If $\gamma$ is not contained in the zero set of $p$, it intersects the zero set of $p$ at most $b \cdot \deg p$ times. \end{corollary}

Now, in analogy to the notion of critical lines, we say that a curve contained in the zero set $Z$ of a polynomial in $\R[x,y,z]$ is a \textit{critical curve} if all the points of the curve are critical points of $Z$. We have shown in \cite{Iliopoulou_12} that the following holds.

\begin{lemma}\label{4.1.14} The zero set of a polynomial $p \in \R[x,y,z]$ contains at most $(\deg p)^2$ critical irreducible real algebraic curves of $\R^3$. \end{lemma}

On a different subject, it is known (see \cite{BCR87} or \cite[Chapter 5]{MR2248869}) that each real semi-algebraic set (i.e., any set of the form $\{x\in \R^n: P(x)=0$ and $Q(x)>0, \; \forall \; Q\in \mathcal{Q} \}$, where $P\in \R[x_1,...,x_n]$ and $\mathcal{Q}$ is a finite family of polynomials in $\R[x_1,...,x_n]$) is the finite, disjoint union of path-connected components. We observe the following.

\begin{lemma}\label{4.1.15} A real algebraic curve in $\R^n$ is the finite, disjoint union of $\lesssim_{b,n} 1$ path-connected components. \end{lemma}

\begin{proof} This is obvious by a closer study of the algorithm in \cite[Chapter 5]{MR2248869} that constitutes the proof of the fact that every real semi-algebraic set is the finite, disjoint union of path-connected components.

\end{proof}

Finally, we are interested in curves in $\R^3$ parametrised by $t \rightarrow \big(p_1(t)$, $p_2(t)$, $p_3(t)\big)$ for $t \in \R$, where $p_i \in \R[t]$ for $i=1,2,3$. Note that, although curves in $\C^3$ with a polynomial parametrisation are, in fact, complex algebraic curves of degree equal to the maximal degree of the polynomials realising the parametrisation (see \cite[Chapter 3, \S3]{Cox+Others/1991/Ideals}), curves in $\R^3$ with a polynomial parametrisation are not, in general, real algebraic curves, which is why we treat their case separately.

More particularly, if a curve $\gamma$ in $\R^3$ is parametrised by $t \rightarrow \big(p_1(t)$, $p_2(t)$, $p_3(t)\big)$ for $t \in \R$, where the $p_i \in \R[t]$, for $i=1,2,3$, are polynomials not simultaneously constant, then the complex algebraic curve $\gamma_{\C}$ parametrised by the same polynomials viewed as elements of $\C[t]$ is irreducible (it is easy to see that if it contains a complex algebraic curve, then the two curves are identical). Therefore, by B\'ezout's theorem, $\gamma_{\C}$ is the unique complex algebraic curve containing $\gamma$.

Taking advantage of this fact, we will show here that each curve in $\R^3$ with a polynomial parametrisation is contained in a real algebraic curve in $\R^3$.

To that end, we first show the following.

\begin{lemma} \label{maybelast}The intersection of a complex algebraic curve in $\C^n$ with $\R^n$ is a real algebraic set, of dimension at most 1.
\end{lemma}

\begin{proof} Let $\gamma_{\C}$ be a complex algebraic curve in $\C^n$, and $\gamma$ the intersection of $\gamma_{\C}$ with $\R^n$. We show that $\gamma$ is a real algebraic set, of dimension at most 1.

Indeed, since $\gamma_{\C}$ is a complex algebraic set in $\C^n$, there exist polynomials $p_1, \ldots,p_k \in \C[x_1,\ldots,x_n]$, such that $\gamma_{\C}$ is the intersection of the zero sets of $p_1, \ldots, p_k$ in $\C^n$. Now, for $i=1,\ldots,k$, the intersection of the zero set of the polynomial $p_i$ in $\C^n$ with $\R^n$ is equal to the zero set of $p_i$ in $\R^n$, which is the same as the zero set of the polynomial $p_i\bar{p_i} \in \R[x_1,\ldots,x_n]$ in $\R^n$. Therefore, $\gamma$ is the intersection of the zero sets of the polynomials $p_1\bar{p_1}, \ldots, p_k\bar{p_k} \in \R[x_1,\ldots,x_n]$ in $\R^n$, it is thus a real algebraic set.

Moreover, let $\prec$ be a term order on the set of monomials in the variables $x_1, \ldots, x_n$. Since $\gamma_{\C}$ is a complex algebraic curve in $\C^n$, it holds that, for every $i\neq j$, $i,j \in \{1,\ldots,n\}$, there exists a monomial in the variables $x_i$ and $x_j$ in the ideal $in_{\prec}(I(\gamma_{\C}))$. Now, if a polynomial $p \in \C[x_1,\ldots,x_n]$ belongs to $I(\gamma_{\C})$, i.e. vanishes on $\gamma_{\C}$, then the polynomial $p \bar{p} \in \R[x_1,\ldots,x_n]$ vanishes on $\gamma$, and thus belongs to the ideal $I(\gamma)$ of $\R[x_1,\ldots,x_n]$. Therefore, there exists a monomial in the variables $x_i$ and $x_j$ in the ideal $in_{\prec}(I(\gamma))$. Consequently, the algebraic set $\gamma$ has dimension at most 1. 

\end{proof}

\begin{corollary} \label{parametrisedcurves} Let $\gamma$ be a curve in $\R^3$, parametrised by $t \rightarrow \big(p_1(t)$, $p_2(t)$, $p_3(t)\big)$ for $t \in \R$, where the $p_i \in \R[t]$, for $i=1,2,3$, are polynomials not simultaneously constant, of degree at most $b$. Then, $\gamma$ is contained in an irreducible real algebraic curve in $\R^3$, of degree at most $b$.
\end{corollary}

\begin{proof} Let $\gamma_{\C}$ be the curve in $\C^3$, parametrised by $t \rightarrow \big(p_1(t), p_2(t),p_3(t)\big)$ for $t \in \C$. As we have already discussed, $\gamma_{\C}$ is the (unique) irreducible complex algebraic curve containing $\gamma$.

Clearly, $\gamma$ is contained in the intersection of $\gamma_{\C}$ with $\R^3$, which, by Lemma \ref{maybelast}, is a real algebraic set, of dimension at most 1. However, since $\gamma_{\C} \cap \R^3$ contains the parametrised curve $\gamma$, it has, in fact, dimension equal to 1. Therefore, $\gamma$ is contained in the real algebraic curve $\gamma_{\C}\cap\R^3$. In fact, $\gamma_{\C}\cap\R^3$ is an irreducible real algebraic curve. Indeed, if $\gamma ' \subsetneq \gamma_{\C}\cap\R^3$ was an irreducible real algebraic curve, and $\gamma_{\C}'$ was the (unique) irreducible complex algebraic curve containing it, then $\gamma_{\C}\cap \R^3 \supsetneq \gamma'=\gamma_{\C}'\cap \R^3$, and thus $\gamma_{\C}\cap \gamma_{\C}'$ $\subsetneq \gamma_{\C}$ would be a complex algebraic curve, which cannot hold, since $\gamma_{\C}$ is irreducible.

Moreover, since $\gamma_{\C}$ is an irreducible algebraic curve in $\C^3$, it is the smallest complex algebraic curve containing the real algebraic curve $\gamma_{\C} \cap \R^3$, and thus the degree of $\gamma_{\C}\cap\R^3$ is equal to $\deg \gamma_{\C}=\max\{\deg p_1, \deg p_2, \deg p_3\}$, and thus equal to at most $b$.

\end{proof}

Note that the above analysis regarding curves parametrised by polynomials appears in \cite{Iliopoulou_13}.

\section{Transversality of more general curves} \label{section6}

We consider the family $\mathcal{F}_n$ of all non-empty sets in $\R^n$ with the property that, if $\gamma \in \mathcal{F}_n$ and $x \in \gamma$, then a basic neighbourhood of $x$ in $\gamma$ is either $\{x\}$ or the finite union of parametrised curves, each homeomorphic to a semi-open line segment with one endpoint the point $x$. In addition, if there exists a parametrisation $f:[0,1)\rightarrow \R^n$ of one of these curves, with $f(0)=x$ and $f'(0)\neq 0$, then the line in $\R^n$ passing through $x$ with direction $f'(0)$ is tangent to $\gamma$ at $x$. If $\Gamma \subset \mathcal{F}_n$, we denote by $T_x^{\Gamma}$ the set of directions of all tangent lines at $x$ to the sets of $\Gamma$ passing through $x$ (note that $T_x^{\Gamma}$ might be empty and that there might exist many tangent lines to a set of $\Gamma$ at $x$).

\begin{definition} Let $\Gamma_1, \ldots, \Gamma_n$ be collections of sets in $\mathcal{F}_n$. A point $x$ in $\R^n$ is a \emph{multijoint} formed by these collections if

\emph{(i)} $x$ belongs to at least one of the sets in $\Gamma_i$, for all $i=1,\ldots,n$, and \newline
\emph{(ii)} there exists at least one vector $v_i$ in $T_x^{\Gamma_i}$, for all $i=1,\ldots,n$, such that the set $\{v_1,\ldots,v_n\}$ spans $\R^n$.

We denote by $J(\Gamma_1,\ldots,\Gamma_n)$ the set of multijoints formed by $\Gamma_1, \ldots, \Gamma_n$.
\end{definition}

We will show here that, under certain assumptions on the properties of the sets in finite collections $\Gamma_1$, $\Gamma_2$ and $\Gamma_3$ in $\mathcal{F}_3$, the corresponding statements of Theorem \ref{theoremmult2} and Theorem \ref{multsimple} still hold.

\begin{definition} Let $\Gamma_1$, $\Gamma_2$, $\Gamma_3$ be collections of real algebraic curves in $\R^3$. We say that the collections $\Gamma_1$, $\Gamma_2$, $\Gamma_3$ are \emph{transversal} if, whenever $\gamma_i \in \Gamma_i$, for $i=1,2,3$, there exist vectors $v_1 \in T_x^{\gamma_1}$, $v_2 \in T_x^{\gamma_2}$ and $v_3 \in T_x^{\gamma_3}$ that span $\R^3$.

\end{definition}

\begin{definition} Let $\Gamma_1$, $\Gamma_2$, $\Gamma_3$ be finite collections of real algebraic curves in $\R^3$.

For all $(N_1,N_2,N_3) \in \R_+^3$, we define as $J_{N_1,N_2,N_3}(\Gamma_1,\Gamma_2,\Gamma_3)$ the set of those multijoints $x$ formed by $\Gamma_1$, $\Gamma_2$, $\Gamma_3$, with the property that there exist transversal collections $\Gamma_1(x) \subseteq \Gamma_1$, $\Gamma_2(x) \subseteq \Gamma_2$ and $\Gamma_3(x) \subseteq \Gamma_3$ of lines passing through $x$, such that $|\Gamma_i(x)|\geq N_i$, for all $i=1,2,3$.

\end{definition}

Indeed, thanks to the results of Section \ref{section5}, we are now ready to formulate and prove Theorem \ref{guthweakcurves} that follows, an extension of Theorem \ref{multsimple}. An immediate corollary is Theorem \ref{theoremmult3}, which is an extension of Theorem \ref{theoremmult2}, and which, by the probabilistic argument we described in Section \ref{section4}, is in fact equivalent to Theorem \ref{guthweakcurves}.

\begin{theorem} \label{guthweakcurves} Let $b$ be a positive constant, and $\Gamma_1$, $\Gamma_2$, $\Gamma_3$ finite collections of real algebraic curves in $\R^3$, of degree at most $b$. Then,
\begin{displaymath}  |J_{N_1,N_2,N_3}(\Gamma_1,\Gamma_2,\Gamma_3)|\leq c_b \cdot \frac{(|\Gamma_1||\Gamma_2||\Gamma_3|)^{1/2}}{(N_1N_2N_3)^{1/2}}, \; \forall\;(N_1,N_2,N_3) \in\N_{+}^3,\end{displaymath}
where $c_b$ is a constant depending only on $b$.
\end{theorem}

\begin{theorem} \label{theoremmult3} Let $b$ be a positive constant, and $\Gamma_1$, $\Gamma_2$, $\Gamma_3$ finite collections of real algebraic curves in $\R^3$, of degree at most $b$. Then,
\begin{displaymath}|J(\Gamma_1,\Gamma_2,\Gamma_3)| \leq c_b \cdot(|\Gamma_1||\Gamma_2||\Gamma_3|)^{1/2}, \end{displaymath}
where $c_b$ is a constant depending only on $b$.
\end{theorem}

\begin{proof} It holds that $J(\Gamma_1,\Gamma_2,\Gamma_3)=J_{1,1,1}(\Gamma_1,\Gamma_2,\Gamma_3)$, therefore the Theorem is an immediate application of Theorem \ref{guthweakcurves} for $(N_1,N_2,N_3)=(1,1,1)$.

\end{proof}

\textit{Proof of Theorem \ref{guthweakcurves}.} The proof will be achieved by induction on the cardinalities of $\Gamma_1$, $\Gamma_2$ and $\Gamma_3$. Indeed, fix $(M_1,M_2,M_3) \in {\N}_+^3$. For $c_b$ an explicit constant $\geq b^{2}$, which depends only on $b$ and will be specified later:

(i) For any finite collections $\Gamma_1$, $\Gamma_2$ and $\Gamma_3$ of real algebraic curves in $\R^3$, of degree at most $b$, such that $|\Gamma_1|=|\Gamma_2|=|\Gamma_3|=1$, we have that
\begin{displaymath} |J_{N_1,N_2,N_3}(\Gamma_1,\Gamma_2,\Gamma_3)|  \leq c_b \cdot \frac{(|\Gamma_1||\Gamma_2||\Gamma_3|)^{1/2}}{(N_1N_2N_3)^{1/2}}, \; \forall\;(N_1,N_2,N_3)\in\R_{+}^3.\end{displaymath}
This is obvious, in fact, for any $c_b \geq b^{2}$, as in this case $|J_{N_1,N_2,N_3}(\Gamma_1,\Gamma_2,\Gamma_3)|=0$ for all $(N_1,N_2,N_3)$ in $\R^3_+$ such that $N_i \gneq 1$ for some $i \in \{1,2,3\}$, while, for $(N_1,N_2,N_3)$ in $\R^3_+$ such that $N_i \leq 1$ for all $i \in \{1,2,3\}$, $|J_{N_1,N_2,N_3}(\Gamma_1,\Gamma_2,\Gamma_3)|$ is equal to at most the number of intersections between the curve in $\Gamma_1$ and the curve in $\Gamma_2$, and thus equal to at most $b^2$.

(ii) Suppose that \begin{displaymath} |J_{N_1,N_2,N_3}(\Gamma_1,\Gamma_2,\Gamma_3)|  \leq c_b \cdot \frac{(|\Gamma_1||\Gamma_2||\Gamma_3|)^{1/2}}{(N_1N_2N_3)^{1/2}}, \; \forall\;(N_1,N_2,N_3) \in \R_{+}^3,\end{displaymath} for any finite collections $\Gamma_1$, $\Gamma_2$ and $\Gamma_3$ of real algebraic curves in $\R^3$, of degree at most $b$, such that $|\Gamma_1|\lneq M_1$, $|\Gamma_2| \lneq M_2$ and $|\Gamma_3| \lneq M_3$.

(iii) We will prove that  \begin{displaymath} |J_{N_1,N_2,N_3}(\Gamma_1,\Gamma_2,\Gamma_3)|  \leq c \cdot \frac{(|\Gamma_1||\Gamma_2||\Gamma_3|)^{1/2}}{(N_1N_2N_3)^{1/2}}, \; \forall\;(N_1,N_2,N_3) \in \R_{+}^3,\end{displaymath} for any finite collections $\Gamma_1$, $\Gamma_2$ and $\Gamma_3$ of real algebraic curves in $\R^3$, of degree at most $b$, such that $|\Gamma_j|= M_j$ for some $j \in \{1,2,3\}$ and $|\Gamma_i| \lneq M_i$, $|\Gamma_k| \lneq M_k$ for $\{i,k\}=\{1,2,3\}\setminus \{j\}$.

Indeed, fix such collections $\Gamma_1$, $\Gamma_2$ and $\Gamma_3$ of real algebraic curves, and $(N_1,N_2,N_3) \in\R_{+}^3$.

For all $j=1,2,3$, each $\gamma \in \Gamma_j$ consists of at most $b \lesssim_b 1$ irreducible components. Therefore, we can assume that each $\gamma \in \Gamma_j$ is an irreducible real algebraic curve, for all $j=1,2,3$.

For simplicity, let $$\mathfrak{G}:= J_{N_1,N_2,N_3}(\Gamma_1,\Gamma_2,\Gamma_3)$$and$$S:=|J_{N_1,N_2,N_3}(\Gamma_1,\Gamma_2,\Gamma_3)|.$$

Now, the proof is completely analogous to that of Theorem \ref{multsimple}. The main differences lie at the beginning and the cellular case, we thus go on to point them out.

We assume that \begin{displaymath} \frac{|\Gamma_1|}{\lceil N_1 \rceil} \leq \frac{|\Gamma_2|}{\lceil N_2\rceil} \leq \frac{|\Gamma_3|}{\lceil N_3 \rceil}. \end{displaymath}

By the definition of the set $\mathfrak{G}$, each point of $\mathfrak{G}$ lies in $\geq \lceil N_1 \rceil$ curves of $\Gamma_1$ and $\geq \lceil N_2\rceil$ curves of $\Gamma_2$. Thus, the quantity $S\lceil N_1\rceil\lceil N_2\rceil$ is equal to at most the number of pairs of the form $(\gamma_1,\gamma_2)$, where $\gamma_1 \in \Gamma_1$, $\gamma_2 \in \Gamma_2$ and the curves $\gamma_1$ and $\gamma_2$ pass through the same point of $\mathfrak{G}$. Therefore, $S\lceil N_1\rceil\lceil N_2\rceil$ is equal to at most the number of all the pairs of the form $(\gamma_1,\gamma_2)$, where $\gamma_1 \in \Gamma_1$ and $\gamma_2 \in \Gamma_2$, i.e. to at most $|\Gamma_1||\Gamma_2|$. So,
\begin{displaymath}S\lceil N_1\rceil\lceil N_2\rceil \leq |\Gamma_1||\Gamma_2|,\end{displaymath}
and therefore
\begin{displaymath} \frac{|\Gamma_1||\Gamma_2|}{S\lceil N_1\rceil\lceil N_2\rceil}\geq 1. \end{displaymath}

Thus, $d:=A\frac{|\Gamma|_1|\Gamma|_2}{S\lceil N_1\rceil\lceil N_2\rceil}$ is a quantity $>1$ for $A>1$. We therefore assume that $A>1$, and we will specify its value later. Now, applying the Guth-Katz polynomial method for this $d>1$ and the finite set of points $\mathfrak{G}$, we deduce that there exists a non-zero polynomial $p\in \R[x,y,z]$, of degree $\leq d$, whose zero set $Z$:

(i) decomposes $\R^3$ in $\lesssim d^3$ cells, each of which contains $\lesssim Sd^{-3}$ points of $\mathfrak{G}$, and

(ii) contains six distinct generic planes, each of which contains a face of a fixed cube $Q$ in $\R^3$, such that the interior of $Q$ contains $\mathfrak{G}$ (and each of the planes is generic in the sense that the plane in $\C^3$ containing it intersects the smallest complex algebraic curve in $\C^3$ containing $\gamma$, for all $\gamma \in \Gamma_1 \cup \Gamma_2$);

to achieve this, we first fix a cube $Q$ in $\R^3$, with the property that its interior contains $\mathfrak{G}$ and the planes containing its faces are generic in the above sense. Then, we multiply the polynomials we end up with at each step of the Guth-Katz polynomial method with the same (appropriate) six linear polynomials, the zero set of each of which is a plane containing a different face of the cube, and stop the application of the method when we finally get a polynomial of degree at most $d$, whose zero set decomposes $\R^3$ in $\lesssim d^3$ cells (the set of the cells now consists of the non-empty intersections of the interior of the cube $Q$ with the cells that arise from the application of the Guth-Katz polynomial method, as well as the complement of the cube).

We can assume that the polynomial $p$ is square-free, as eliminating the squares of $p$ does not inflict any change on its zero set.

Let us now assume that there are $\geq 10^{-8}S$ points of $\mathfrak{G}$ in the union of the interiors of the cells; by choosing $A$ to be a sufficiently large constant depending only on $b$, we will be led to a contradiction.

Indeed, there are $\gtrsim S$ points of $\mathfrak{G}$ in the union of the interiors of the cells. However, we also know that there exist $\lesssim d^3$ cells in total, each with $\lesssim Sd^{-3}$ points of $\mathfrak{G}$. Therefore, there exist $\gtrsim d^3$ cells, with $\gtrsim Sd^{-3}$ points of $\mathfrak{G}$ in the interior of each. We call the cells with this property ``full cells".

Now, for every full cell,  let $\mathfrak{G}_{cell}$ be the set of points of $\mathfrak{G}$ in the interior of the cell, $\Gamma_{1,cell}$ and $\Gamma_{2,cell}$ the sets of curves of $\Gamma_1$ and $\Gamma_2$, respectively, each containing at least one point of $\mathfrak{G}_{cell}$, and $S_{cell}:=|\mathfrak{G}_{cell}|$. Now,
\begin{displaymath}S_{cell}\lceil N_1\rceil \lceil N_2\rceil  \lesssim |\Gamma_{1,cell}||\Gamma_{2,cell}|, \end{displaymath}
as the quantity $S_{cell}\lceil N_1\rceil \lceil N_2\rceil $ is equal to at most the number of pairs of the form $(\gamma_1,\gamma_2)$, where $\gamma_1 \in \Gamma_{1,cell}$, $\gamma_2 \in \Gamma_{2,cell}$ and the curves $\gamma_1$ and $\gamma_2$ pass through the same point of $\mathfrak{G}_{cell}$. Thus, $S_{cell}\lceil N_1\rceil \lceil N_2\rceil $ is equal to at most the number of all the pairs of the form $(\gamma_1,\gamma_2)$, where $\gamma_1 \in \Gamma_{1,cell}$ and $\gamma_2 \in \Gamma_{2,cell}$, i.e. to at most $|\Gamma_{1,cell}||\Gamma_{2,cell}|$.

Therefore,
\begin{displaymath}(|\Gamma_{1,cell}||\Gamma_{2,cell}|)^{1/2} \gtrsim S_{cell}^{1/2} (\lceil N_1\rceil\lceil N_2\rceil )^{1/2} \gtrsim \frac{S^{1/2}}{d^{3/2}}(\lceil N_1\rceil\lceil N_2\rceil)^{1/2}. \end{displaymath}
Furthermore, for every full cell and $i \in \{1,2\}$, let $\Gamma_{i,Z}$ be the set of curves of $\Gamma_i$ which are lying in $Z$, and $\Gamma_{i,cell}'$ the set of curves in $\Gamma_{i,cell}$ such that, if $\gamma \in \Gamma_{i,cell}'$, there does not exist any point $x$ in the intersection of $\gamma$ with the boundary of the cell, with the property that the induced topology from $\R^3$ to the intersection of $\gamma$ with the closure of the cell contains some open neighbourhood of $x$. Obviously, $\Gamma_{i,cell} \subset \Gamma_i \setminus \Gamma_{i,Z}$. Finally, let $I_{i,cell}$ denote the number of incidences between the boundary of the cell and the curves in $\Gamma_{i,cell}$.

Now, let $i \in \{1,2\}$. Each of the curves in $\Gamma_{i,cell}\setminus \Gamma'_{i,cell}$ intersects the boundary of the cell at at least one point $x$, with the property that the induced topology from $\R^3$ to the intersection of the curve with the closure of the cell contains an open neighbourhood of $x$; therefore, $I_{i,cell}\geq |\Gamma_{i,cell}\setminus \Gamma'_{i,cell}|$ ($=|\Gamma_{i,cell}|-|\Gamma'_{i,cell}|$). Also, the union of the boundaries of all the cells is the zero set $Z$ of $p$, and if $x$ is a point of $Z$ which belongs to a curve in $\Gamma_i$ intersecting the interior of a cell, such that the induced topology from $\R^3$ to the intersection of the curve with the closure of the cell contains an open neighbourhood of $x$, then there exist at most $2b-1$ other cells whose interior is also intersected by the curve and whose boundary contains $x$, such that the induced topology from $\R^3$ to the intersection of the curve with the closure of each of these cells contains some open neighbourhood of $x$. So, if $I_i$ is the number of incidences between $Z$ and $\Gamma_i \setminus \Gamma_{i,Z}$, and $\mathcal{C}$ is the set of all the full cells (which, in this case, has cardinality $\gtrsim d^3$), then
\begin{displaymath} I_i \geq \frac{1}{2b} \cdot \sum_{cell \;\in \;\mathcal{C}} I_{i,cell} \geq \frac{1}{2b}\cdot \sum_{cell \;\in \;\mathcal{C}}(|\Gamma_{i,cell}|-|\Gamma'_{i,cell}|).\end{displaymath}
Now, if $\gamma \in \Gamma_{i,cell}$ ($\supseteq \Gamma_{i,cell}'$), we consider the (unique, irreducible) complex algebraic curve $\gamma_{\C}$ in $\C^3$ which contains $\gamma$. In addition, let $p_{\C}$ be the polynomial $p$ viewed as an element of $\C[x,y,z]$, and $Z_{\C}$ the zero set of $p_{\C}$ in $\C^3$. The polynomial $p$ was constructed in such a way that $\gamma_{\C}$ intersects each of 6 complex planes, each of which contains one of the real planes in $Z$ that contain the faces of the cube $Q$; consequently $\gamma_{\C}$ intersects $Z_{\C}$ at least once. Moreover, if $\gamma^{(1)}$, $\gamma^{(2)}$ are two distinct curves in $\Gamma_i$, then $\gamma^{(1)}_{\C}$, $\gamma^{(2)}_{\C}$ are two distinct curves in $\Gamma_{\C}$ (since $\gamma^{(1)}=\gamma^{(1)}_{\C}\cap \R^3$, while $\gamma^{(2)}=\gamma^{(2)}_{\C}\cap \R^3$). So, if $\Gamma_{i,\C}=\{\gamma_{\C}: \gamma \in \Gamma_{i,cell}$, for some cell in $\mathcal{C}\}$ and $I_{i,\C}$ denotes the number of incidences between $\Gamma_{i,\C}$ and $Z_{\C}$, it follows that
$$I_{i,\C}\geq |\Gamma_{i,\C}|=|\Gamma_i|\geq \big|\bigcup_{cell \;\in\; \mathcal{C}}\Gamma_{i,cell}'\big|,
$$while also
$$I_{i,\C}\geq I_i.
$$
Therefore,
$$I_{i,\C}\geq \frac{1}{2}(I_i+I_{i,\C})\geq
$$
\begin{displaymath}\geq \frac{1}{2} \cdot \bigg( \frac{1}{2b} \sum_{cell \;\in\; \mathcal{C}}(|\Gamma_{i,cell}|-|\Gamma_{i,cell}'|) + \big|\bigcup_{cell \;\in \;\mathcal{C}}\Gamma_{i,cell}'\big|\bigg) \sim_b \end{displaymath}
\begin{displaymath} \sim_b\sum_{cell \;\in\; \mathcal{C}}(|\Gamma_{i,cell}|-|\Gamma_{i,cell}'|) + \big|\bigcup_{cell \;\in \;\mathcal{C}}\Gamma_{i,cell}'\big|.\end{displaymath}
However, each real algebraic curve in $\R^3$, of degree at most $b$, is the disjoint union of $\leq R_b$ path-connected components, for some constant $R_b$ depending only on $b$ (by Lemma \ref{4.1.15}). Hence,
$$\big|\bigcup_{cell \;\in \;\mathcal{C}}\Gamma_{i,cell}'\big| \sim_b\sum_{cell \;\in \;\mathcal{C}}|\Gamma_{i,cell}'|,
$$
from which it follows that
 \begin{displaymath}I_{i,\C}\gtrsim_b \sum_{cell \;\in\; \mathcal{C}}(|\Gamma_{i,cell}|-|\Gamma_{i,cell}'|)+\sum_{cell \;\in \;\mathcal{C}}|\Gamma_{i,cell}'|\sim_b\end{displaymath}
$$\sim_b \sum_{cell \;\in \;\mathcal{C}}|\Gamma_{i,cell}|.
$$
On the other hand, however, each $\gamma_{\C}\in \Gamma_{i,\C}$ is a complex algebraic curve of degree at most $b$ which does not lie in $Z_{\C}$, and thus intersects $Z_{\C}$ at most $b \cdot \deg p$ times. Therefore,
\begin{displaymath}I_{i,\C} \lesssim_b |\Gamma_i| \cdot d.
\end{displaymath}
So, for $i \in \{1,2\}$, it holds that
$$\sum_{cell \;\in \;\mathcal{C}}|\Gamma_{i,cell}| \lesssim_b |\Gamma_i| \cdot d.
$$
Hence, from all the above we obtain
\begin{displaymath} \sum_{cell \in \mathcal{C}} \frac{S^{1/2}}{d^{3/2}}(\lceil N_1\rceil \lceil N_2\rceil )^{1/2} 
\lesssim_b \sum_{cell \in \mathcal{C}} (|\Gamma_{1,cell}||\Gamma_{2,cell}|)^{1/2} \lesssim_b \end{displaymath}
\begin{displaymath} \lesssim_b\Bigg(\sum_{cell \in\mathcal{C}} |\Gamma_{1,cell}|\Bigg)^{1/2} \Bigg(\sum_{cell \in\mathcal{C}}|\Gamma_{2,cell}|\Bigg)^{1/2}\lesssim_b \end{displaymath}
\begin{displaymath} \lesssim_b (|\Gamma_1|\cdot d)^{1/2}(|\Gamma_2|\cdot d)^{1/2} \sim_b (|\Gamma_1||\Gamma_2|)^{1/2}d.\end{displaymath}

But the full cells number $\gtrsim d^3$. Thus,
$$d^{3/2}S^{1/2}(\lceil N_1\rceil \lceil N_2\rceil )^{1/2} \lesssim_b (|\Gamma_1||\Gamma_2|)^{1/2}d,
$$
which in turn gives $A\lesssim_b 1$. In other words, there exists some constant $C_b$, depending only on $b$, such that $A \leq C_b$. By fixing $A$ to be a number larger than $C_b$ (and of course large enough to have that $d> 1$), we are led to a contradiction.

Therefore, for $A$ a sufficiently large constant that depends only on $b$, it holds that more than $(1-10^{-8})S$ points of $\mathfrak{G}$ lie in the zero set of $p$. 

The rest of the proof follows in a similar way as the proof of Theorem \ref{multsimple}.

\qed

We are now able to count multijoints formed by a finite collection of curves in $\R^3$, parametrised by real univariate polynomials of uniformly bounded degree.

\begin{corollary}\label{polynomialcurves} Let $b$ be a positive constant and $\Gamma_1$, $\Gamma_2$, $\Gamma_3$ finite collections of curves in $\R^3$, such that, for $j=1,2,3$, each $\gamma \in \Gamma_j$ is parametrised by $t \rightarrow \big(p^{\gamma}_1(t), p_2^{\gamma}(t),p_3^{\gamma}(t)\big)$ for $t \in \R$, where the $p_i^{\gamma} \in \R[t]$, for $i=1,2,3$, are polynomials not simultaneously constant, of degrees at most $b$. Then,
\begin{displaymath}|J(\Gamma_1,\Gamma_2,\Gamma_3)| \leq c_b \cdot(|\Gamma_1||\Gamma_2||\Gamma_3|)^{1/2}, \end{displaymath}
where $c_b$ is a constant depending only on $b$.
\end{corollary}

\begin{proof} By Corollary \ref{parametrisedcurves}, each $\gamma \in \Gamma_j$, for $j=1,2,3$, is contained in a real algebraic curve in $\R^3$, of degree at most $b$. Therefore, the statement of the Corollary follows from Theorem \ref{theoremmult3}.

\end{proof}

Again via the probabilistic argument described earlier, we immediately deduce the following.

\begin{corollary} \label{weakpolynomialcurves} Let $b$ be a positive constant and $\Gamma_1$, $\Gamma_2$, $\Gamma_3$ finite collections of curves in $\R^3$, such that, for $j=1,2,3$, each $\gamma \in \Gamma_j$ is parametrised by $t \rightarrow \big(p^{\gamma}_1(t), p_2^{\gamma}(t),p_3^{\gamma}(t)\big)$ for $t \in \R$, where the $p_i^{\gamma} \in \R[t]$, for $i=1,2,3$, are polynomials not simultaneously constant, of degrees at most $b$. Then,
\begin{displaymath}  |J_{N_1,N_2,N_3}(\Gamma_1,\Gamma_2,\Gamma_3)|\leq c_b \cdot \frac{(|\Gamma_1||\Gamma_2||\Gamma_3|)^{1/2}}{(N_1N_2N_3)^{1/2}}, \; \forall\;(N_1,N_2,N_3) \in\N_{+}^3,\end{displaymath}
where $c_b$ is a constant depending only on $b$.
\end{corollary}

\end{document}